\documentclass[10pt]{amsart}
  \usepackage{amsmath,amsfonts,amssymb,amsthm,enumerate,color,mathrsfs,extarrows,graphicx,eucal,longtable}       
  \usepackage[a4paper,margin=1.2in]{geometry}  
  \usepackage[all]{xy}
  \usepackage[utf8]{inputenc}
  \usepackage[colorlinks,linkcolor=black,citecolor=black]{hyperref}

  \newtheorem{thm}{Theorem}[section]
  \newtheorem{prop}[thm]{Proposition}
  \newtheorem{lemma}[thm]{Lemma}
  \newtheorem{cor}[thm]{Corollary}
  \newtheorem{question}[thm]{Question}
  
  \theoremstyle{definition}
  \newtheorem{example}[thm]{Example}
  \newtheorem{defin}[thm]{Definition}
  \newtheorem{remark}[thm]{Remark}

  \renewcommand{\AA}{\mathbf{A}}
  \newcommand{\RR}{\mathbf{R}}
  \newcommand{\CC}{\mathbf{C}}
  \newcommand{\NN}{\mathbf{N}}
  \newcommand{\ZZ}{\mathbf{Z}}
  \newcommand{\XX}{\mathbf{X}}
  
  \newcommand{\PP}{\mathbf{P}}

  \newcommand{\Xx}{\mathscr{X}}
  \newcommand{\Mm}{\mathscr{M}}
  \newcommand{\cO}{\mathscr{O}}
  \newcommand{\Cc}{\mathscr{C}}

  \newcommand{\et}{{\rm \acute{e}t}}
  \newcommand{\ket}{\mathrm{k \acute{e}t}}

  \DeclareMathOperator{\Spec}{Spec}
  \DeclareMathOperator{\Spf}{Spf}
  \DeclareMathOperator{\cd}{cd}
  \newcommand{\Sone}{{\mathbf{S}^1}}

  \newcommand{\Spaces}{\mathbf{Spaces}}
  \newcommand{\SpSone}{\Spaces_{/\Sone}}

  \newcommand{\gp}{\mathrm{gp}}
  \renewcommand{\log}{\mathrm{log}} 
  \newcommand{\an}{\mathrm{an}}
  \newcommand{\triv}{\mathrm{triv}}
  \newcommand{\Hom}{\mathrm{Hom}}

  \newcommand{\Div}{\mathrm{Div}}
  \DeclareMathOperator{\hocolim}{hocolim}
  \DeclareMathOperator{\holim}{holim}
  \DeclareMathOperator{\cosk}{cosk}
  
  \DeclareMathOperator*{\colim}{colim}

  \newcommand{\CCt}{{\CC(\!(t)\!)}}
  \newcommand{\CCs}{{\CC[\![t]\!]}}

  \newcommand{\CCq}{{\CC(\!(q)\!)}}
  \newcommand{\CCsq}{{\CC[\![q]\!]}}

  \newcommand{\Perf}{\mathrm{Perf}}

  \renewcommand{\to}{%
    \mathchoice
      {\longrightarrow} % \displaystyle
      {\rightarrow} % \textstyle
      {\rightarrow} % \scriptstyle
      {\rightarrow} % \scriptscriptstyle
  }

  \newcommand*\isomlong{%
    \xlongrightarrow{\raisebox{-0.2 em}{\smash{\ensuremath{\sim}}}}%
  }

  \title[Betti realization of varieties defined by formal Laurent series]{Betti realization of varieties\\ defined by formal Laurent series}
  \date{\today}

  \author[P.\ Achinger]{Piotr Achinger}
  \address{Instytut Matematyczny PAN, Śniadeckich 8, Warsaw, Poland}
  \email{pachinger@impan.pl}

  \author[M.\ Talpo]{Mattia Talpo}
  \address{Dipartimento di Matematica, Universit\`a di Pisa, Largo Bruno Pontecorvo 5, 56127 Pisa, Italy}  \email{mattia.talpo@unipi.it}

  \keywords{Betti realization, topology of degenerations, log geometry, Kato--Nakayama space, \'etale homotopy, rigid analytic space}
  \subjclass[2010]{14D06, 14F35, 14F45}

\begin{document}

\begin{abstract}
  We give two constructions of functorial topological realizations for schemes of finite type over the field $\CCt$ of formal Laurent series with complex coefficients, with values in the homotopy category of spaces over the circle. The problem of constructing such a realization was stated by D.\ Treumann, motivated by certain questions in mirror symmetry. The first construction uses spreading out and the usual Betti realization over $\CC$. The second uses generalized semistable models and log Betti realization defined by Kato and Nakayama, and applies to smooth rigid analytic spaces as well. We provide comparison theorems between the two constructions and relate them to the \'etale homotopy type and de Rham cohomology. As an illustration of the second construction, we treat two examples, the Tate curve and the non-archimedean Hopf surface.
\end{abstract}

\maketitle

\section{Introduction}

The study of a complex algebraic variety $X$ is facilitated by the existence of the associated complex analytic space $X_\an$. Its homotopy type, sometimes called the Betti realization, encodes important topological invariants such as
\[ 
  H_i(X_\an, \ZZ), \quad H^i(X_\an, \ZZ), \quad \pi_i(X_\an), \quad \ldots
\]
For a general (locally noetherian) scheme $X$, one usually uses the \'etale topology instead, and the homotopy type thus obtained (the \'etale pro-homotopy type of Artin--Mazur \cite{ArtinMazur}) encodes the invariants
\[ 
  H^i(X_\et, \ZZ/N\ZZ), \quad \pi_i^\et(X), \quad \ldots
\]
In the case of a normal complex algebraic variety, the \'etale homotopy type is the pro-finite completion of the classical homotopy type, and might contain less information.

In this paper, we deal with schemes and rigid analytic spaces over the field $\CCt$ of formal Laurent series over $\CC$, or more generally over any complete discretely valued field $K$ whose residue field $k$ is equipped with an embedding $\iota \colon k\hookrightarrow \CC$ (e.g. $K=\mathbb{Q}(\!(t)\!)$). Our main result is the construction of a functorial topological realization of all schemes locally of finite type over $K$, allowing one to define integral invariants as above together with the monodromy action of $\pi_1(\Sone)$. This answers the question of David Treumann \cite{TreumannMO,TreumannNotes}, motivated by a problem related to the complex topological $K$-theory of mirror pairs. 

\subsection{Motivating example}
To give the reader a feeling for the problem and its possible solutions, let us consider the elliptic curve $E$ over $\CCt$ given by the equation
\begin{equation} \label{eqn:ell-curve} 
  y^2 = x^3 + Ax+ B, \quad A, B\in \CCt, \quad \Delta = 4A^3 + 27B^2 \neq 0.
\end{equation}
We want to attach to $E$ a two-dimensional real torus, canonically defined up to homotopy. (Giving a torus is equivalent to giving its first homology group, so the problem is equivalent to defining a rank two lattice $H_1(E, \ZZ)$.)
If $A$ and $B$ both have positive radii of convergence (for example, if they are algebraic over $\CC(t)$), then \eqref{eqn:ell-curve} gives a holomorphic family of elliptic curves $E_t$ over a punctured disc $\Delta^* = \{0<|t|<\varepsilon\}$. In this case, we can define the homotopy type of $E$ to be any fiber of this family (say, for $t = \varepsilon/2$), or better the restriction of the family to the circle $|t|=\varepsilon/2$ treated as a space over the circle (to keep track of the monodromy). This can be seen to be independent up to homotopy of the choice of an equation \eqref{eqn:ell-curve} defining~$E$.

In the general case, let $R = \CC[t, A, B][\Delta^{-1}] \subseteq \CCt$ be the ring generated by the coefficients, with the discriminant inverted. Then \eqref{eqn:ell-curve} gives a family $\mathscr{E}$ of elliptic curves over $S = \Spec R$. Now $S$ is a scheme of finite type over $\CC$, and looking at the associated analytic spaces we obtain a~fibration in tori $\mathscr{E}_{\rm an} \to S_{\rm an}$. A little bit of thought shows that there is still a preferred homotopy class of maps $s\colon \Sone\to S_{\rm an}$ (``approximating'' the given map $\Spec \CCt\to S$), and we can define the homotopy type of $E$ to be the base change of $\mathscr{E}_{\rm an}$ along $s$. 

What helped us in the above approach is the fact that $E$, being a scheme of finite type over $\CCt$, admits a model over some finitely generated $\CC[t]$-subalgebra. Since such ``spreading out'' is not available in the context of rigid analytic spaces, we need a different approach (which turns out to be useful in the case of schemes as well). So, let us consider the rigid analytic space $\mathcal{E}$ associated to $E$ (the reader unfamiliar with non-archimedean geometry can keep considering $E$). The role of spreading out will be played by a choice of a semistable formal model $\mathfrak{E}$ over $\CCs$. The machinery of logarithmic geometry of Fontaine, Illusie, and Kato, with important contributions of Nakayama and Ogus, allows one to treat the special fiber $\mathfrak{E}_0$ at $t=0$, endowed with the induced logarithmic structure, as a smooth member of the family. A surprising ``real oriented blowup'' construction due to Kato and Nakayama functorially attaches to $\mathfrak{E}_0$ (or any fs log scheme over $\CC$) a topological space $\mathfrak{E}_{0, \rm log}$, which in this case turns out to be a fibration in $2$-tori over the circle $\Sone$ (see Figure~\ref{fig:dwork}). One should think of the base $\Sone$ as the ``circle of radius $0$'' in the complex plane. 

By the birational geometry of surfaces, any two such models are related by a zigzag of blowups of points in the special fiber. A direct calculation of the fibers of the map $\mathfrak{E}'_{0,\rm log}\to \mathfrak{E}_{0,\rm log}$ induced by such a blowup shows that it is a homotopy equivalence, giving the independence up to homotopy of $\mathfrak{E}_{0, \rm log}$ of the formal model $\mathfrak{E}$.

\begin{figure}[h]
  \centering
  \includegraphics[width=.8\textwidth]{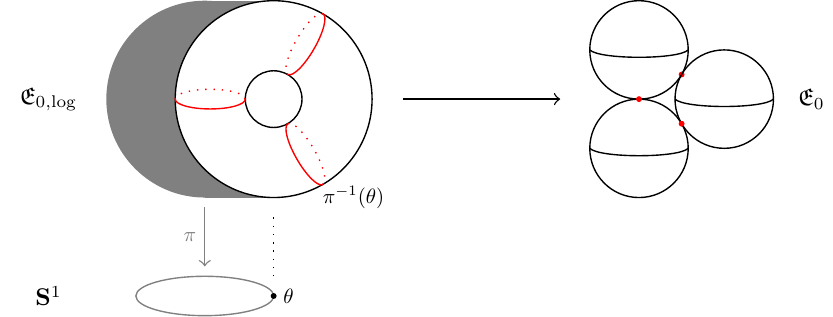} 
  \caption{The Kato--Nakayama space of the special fiber of a semistable model $\mathfrak{E}$ is a fibration in tori over $\Sone$. The nodes in the special fiber $\mathfrak{E}_0$ become loops on the torus.}  
  \label{fig:dwork}
\end{figure}

In order to generalize this example to all schemes and smooth rigid analytic spaces, we will need to study models of schemes over $\CCt$ over finitely generated $\CC$-subalgebras, as well as good formal models of rigid analytic spaces over $\CCt$ and the Kato--Nakayama spaces of their special fibers. The independence of the model in the latter case requires more sophisticated birational geometry supplied by the Weak Factorization Theorem.

\subsection{The construction for schemes}
The first goal of this paper is to show that the above constructions can be generalized to define a well-behaved functor defined for all schemes locally of finite type over the field $K$. We can summarize our findings as follows.

\begin{thm}
Let $K$ be a complete discretely valued field whose residue field $\cO_K/\mathfrak{m}=k$ is endowed with an embedding $\iota\colon k\hookrightarrow \CC$, and denote by $(-)_0$ the base change along $\iota$. There exists a functor $\Psi$ from the category of schemes locally of finite type over $K$ to the $\infty$-category of spaces over $\Sone$, enjoying the following properties:

\medskip 

\begin{enumerate}
  \item (Finiteness) If $X$ is separated and of finite type, then $\Psi(X)$ has the homotopy type of a~fibration in finite CW-complexes over $\Sone$.

  \medskip

  \item (Sheaf property) The functor $\Psi$ satisfies homotopical descent with respect to the $h$-topology: if $Y_\bullet\to X$ is an $h$-hypercovering, then the induced map
  \[ 
    \hocolim \Psi(Y_\bullet) \to \Psi(X)
  \]
  is an equivalence.

  \medskip

  \item (Good reduction) If $X$ is the general fiber of a smooth and proper scheme $\mathscr{X}$ over the valuation ring $\cO_K$, then there is a natural identification 
  \[
    \Psi(X) \simeq \mathscr{X}_{0, \an} \times \Sone.
  \]

  \medskip
  
  \item (Semistable reduction, with boundary) Let $\mathscr{Y}$ be a proper, flat, and regular scheme over $\cO_K$ and let $\mathscr{D}\subseteq \mathscr{Y}$ be a divisor with normal crossings containing the reduced special fiber $(\mathscr{Y}_{k})_{\rm red}$. Let $X = \mathscr{Y}\setminus \mathscr{D}$. Then there is a natural identification
  \[ 
    \Psi(X) \simeq \mathscr{Y}_{0, \rm log}
  \]
  where $\mathscr{Y}_{0, \rm log}$ is the Kato--Nakayama space of the base changed central fiber $\mathscr{Y}_0$, with its natural log structure induced by the pair $(\mathscr{Y}, \mathscr{D})$, treated as a space over $0_{\rm log} = \Sone$.

  \medskip
   
  \item (\'Etale comparison) Suppose that the residue field $k$ of $K$ is algebraically closed. Then there is a natural identification
  \[ 
    \widehat{\Psi}(X)  \simeq \widehat{\Pi}^\et(X)
  \]
  between the profinite completion of $\Psi(X)$ and of the \'etale homotopy type of $X$. 
  
  \medskip
  
  \item (De Rham comparison) Fix an isomorphism $K\simeq k(\!(t)\!)$. For $X$ of finite type, there is a functorially defined finite free graded $\cO_K$-module $H^*(X)$, endowed with a logarithmic connection $\nabla$ and a Griffiths-transverse Hodge filtration, and canonical isomorphisms
  \[ 
    H^*(X)\otimes_{\cO_K} K \simeq H^*_{\rm dR}(X/K), \quad H^*(X)\otimes_{\cO_K} \CC \simeq H^*(\widetilde\Psi(X), \CC),
  \]
  where the first isomorphism respects the connections and Hodge filtrations, and the second one identifies the monodromy operator with $\exp (-2\pi i \,{\rm res}_0 \nabla)$. Here $\widetilde\Psi(X)$ denotes the fiber of $\Psi(X)$ over a point of $\Sone$. 
 \end{enumerate} 
\end{thm}

The functor is defined using a spreading out procedure as in the motivating example, which \emph{a~priori} depends on a choice of an embedding $K\hookrightarrow \CCt$ lifting the given $\iota$. Property (1) is then easy to obtain. Generalizing results of Dugger--Isaksen and Blanc about homotopical descent for topological spaces, we show the $h$-descent property (2). This in turn, by Nagata compactification and resolution of singularities, allows us to reduce the computation of $\Psi(X)$ to the case when $X$ is as in (4); results of Nakayama and Ogus imply property (4). \emph{A posteriori}, we see that the functor $(\mathscr{Y}, \mathscr{D}) \mapsto \mathscr{Y}_{0, \rm log}$ satisfies a version of $h$-descent, and that the functor $\Psi$ depends only on the choice of $\iota$ (recall \cite{SerreConj} that for a scheme of finite type over a field $k\subseteq \CC$, the homotopy type of $X(\CC)$ may depend on the choice of the embedding of $k$ into $\CC$). We prove (5) and (6) using the log geometry description of (4). 

Summarizing, the functor is easy to define using spreading out, but easier to compute and compare with other objects using log geometry.

\subsection{Rigid analytic spaces}

Our second goal is to extend the above construction to smooth rigid analytic spaces over $K$. Our construction uses formal models as in the motivating example, and therefore the approach to rigid geometry due to Raynaud \cite{Raynaud,FormalRigidI} is the most appropriate.

\begin{thm}
Let $K$ be as above. There exists a functor $\Psi_{\rm rig}$ from the category smooth rigid analytic spaces over $K$ to the $\infty$-category of spaces over $\Sone$, enjoying the following properties:

  \medskip
  
\begin{enumerate}
  \item (Finiteness) If $\mathcal{X}$ is separated and quasi-compact, or more generally is of the form $\mathcal{Y}\setminus \mathcal{Z}$ where $\mathcal{Y}$ is quasi-compact and separated and $\mathcal{Z}\subseteq\mathcal{Y}$ is a closed analytic subset, then $\Psi_{\rm rig}(\mathcal{X})$ has the homotopy type of a finite CW-complex.

  \medskip
  
  \item (Sheaf property) The functor $\Psi_{\rm rig}$ satisfies homotopical descent with respect to admissible coverings.

  \medskip
  
  \item (Good reduction) If $\mathcal{X}$ is the general fiber of a smooth formal scheme $\mathfrak{X}$ over the valuation ring $\cO_K$, then there is a natural identification 
  \[
    \Psi_{\rm rig}(\mathcal{X}) \simeq \mathfrak{X}_{0,\an} \times \Sone.
  \]

  \medskip
  
  \item (Semistable reduction) Let $\mathfrak{X}$ be a flat and regular formal scheme, separated and of finite type over $\cO_K$, such that the reduced special fiber $(\mathfrak{X}_{k})_{\rm red}$ is a divisor with normal crossings. Let $\mathcal{X} = \mathfrak{X}_{\rm rig}$ be its generic fiber, which is a smooth quasi-compact and separated rigid analytic space. Then there is a natural identification
  \[ 
    \Psi_{\rm rig}(\mathcal{X}) \simeq \mathfrak{X}_{0, \rm log}
  \]
  where $\mathfrak{X}_{0, \rm log}$ is the Kato--Nakayama space of the base changed special fiber $\mathfrak{X}_0$ with its natural log structure induced by the pair $(\mathfrak{X}, \mathfrak{X}_{k})$, treated as a space over $0_{\rm log} = \Sone$.

  \medskip
  
  \item (Comparison for schemes) If $X$ is a smooth scheme over $K$, with associated rigid analytic space $X_{\rm an}$, then there is a natural identification
  \[ 
    \Psi(X) \simeq \Psi_{\rm rig}(X_{\rm an}).
  \]
\end{enumerate}
\end{thm} 

Here, the functor is actually \emph{defined} using (4). One needs to show directly that this is independent of the choice of a model $\mathfrak{X}$ (this step was easy with spreading out). We deal with this using a version of the Weak Factorization Theorem due to Abramovich and Temkin \cite{weakfactorization}, which turns the statement into a computation of the fibers of the map $\mathfrak{X}'_{0, \rm log}\to \mathfrak{X}_{0, \rm log}$ induced by a~simple admissible blowing up $\mathfrak{X}'\to \mathfrak{X}$. Assertion (5) is not directly obvious, as the analytification of a~non-proper algebraic variety is not quasi-compact; to this end, we prove a general ``Purity Theorem'' (Theorem~\ref{thm:purity}) comparing the homotopy types associated to a quasi-compact rigid snc pair $(\mathcal{Y}, \mathcal{D})$ and the non-quasi compact complement $\mathcal{X} = \mathcal{Y}\setminus \mathcal{D}$.

We illustrate the construction with two classical examples: the Tate curve and the non-archimedean Hopf surface.

We do not know if $\Psi_{\rm rig}$ satisfies $h$-descent (Question~\ref{question:descent}). Using resolution of singularities, this property would allow one to extend $\Psi_{\rm rig}$ to all rigid analytic spaces.

\subsection{Comparison with other work}

A stable version of the functor $\Psi$ has been constructed previously by Ayoub \cite{ayoub_2010}. In \cite{stewart-vologodsky} Stewart and Vologodsky use the Kato--Nakayama space of the central fiber of a semistable model to define a mixed Hodge structure for a smooth projective variety over $\CCt$ (see also \S\ref{sec:dr}). The cohomology groups $H^i(\mathfrak{X}, \mathbf{Z})$ for a semistable formal scheme $\mathfrak{X}$ over $\CCs$ have been independently considered by Berkovich \cite{berkovich} using similar methods.

\subsection{Outline of the paper}

In Section~\ref{sec:spreading}, we prove the existence of topologically nice models of schemes of finite type over $\CCs$ defined over suitable finitely generated smooth $\CC[t]$-subalgebras (Corollary~\ref{cor:models}), which allows us to first construct the functor $\Psi$ for separated schemes of finite type over $\CCt$ (\S\ref{sec:construction}). We compare the construction to the natural one over the field of convergent power series $\CC(\!\{t\}\!)$ (\S\ref{sec:conv}). In \S\ref{sec:descent}, we show that the functor $\Psi$ defined thus far satisfies homotopical descent in the $h$-topology, and in particular it naturally extends to all schemes locally of finite type. 

Section~\ref{sec:log.geometry} deals with log geometry and topological properties of Kato--Nakayama spaces. We then construct a functor $\Psi_{\rm log}$ and compare it with the functor $\Psi$ defined previously. In the final \S\ref{sec:dr} we compare the de Rham cohomology of $X/K$ with the singular cohomology of $\Psi(X)$. 

In Section~\ref{sec:etale}, we prove that $\Psi(X)$ agrees with the \'etale homotopy type up to profinite completion (Theorem~\ref{thm:etale-comparison}). To this end, we need to prove an auxiliary result (Theorem~\ref{thm:kummer-csst}), a variant of the results of \cite{knvsroot,loghomotopy}, comparing the Kummer \'etale homotopy type with the homotopy type of the Kato--Nakayama space.

The final Section~\ref{s:rigid} deals with rigid analytic spaces. After some preliminaries on localizations of categories (\S\ref{ss:localiz}) and rigid spaces via formal schemes (\S\ref{ss:raynaud}), we set the stage with suitable categories of good formal models in \S\ref{ss:good-models-rigid}. In \S\ref{sec:key.calculation}, we perform the key calculation of fibers of the maps on Kato--Nakayama spaces induced by ``simple blowups,'' the types of blowups appearing in the Weak Factorization Theorem. In \S\ref{sec:construction-rigid}, we combine this with a theorem of Smale to construct the realization functor $\Psi_{\rm rig}$ for quasi-compact and separated rigid analytic spaces. In the subsequent \S\ref{sec:descent-rigid}, we first show that this functor satisfies homotopical descent with respect to the admissible topology, and hence it extends to all smooth rigid analytic spaces, and then show the purity result (Theorem~\ref{thm:purity}), allowing us to compare $\Psi_{\rm rig}(X_\an)$ with $\Psi(X)$. The final \S\ref{ss:examples}, we study some examples and state two open questions.

\medskip
\noindent {\bf Acknowledgements.} We are grateful to David Treumann for originally posting a question on MathOverflow \cite{TreumannMO} that motivated us to work on this project. We enjoyed useful conversations with Ben Antieau, Bhargav Bhatt, Elden Elmanto, H\'el\`ene Esnault, Tyler Foster, Denis Nardin, Arthur Ogus, Martin Olsson, David Rydh, Karol Szumiło, Vadim Vologodsky, and Olivier Wittenberg. We are especially grateful to David Carchedi and Mauro Porta for their generous help with infinity categories. We thank Joe Berner for bringing \cite{berkovich} to our attention.

\medskip

M.\ T.\ was supported by a PIMS postdoctoral fellowship and by EPSRC grant EP/R013349/1.

P.\ A.\ was supported by NCN SONATA grant number {\tt 2017/26/D/ST-1/00913} and by the ERC Starting Grant 802787 KAPIBARA.

This material is based upon work supported by the National Science Foundation under Grant No.\ 1440140, while P.\ A.\ was in residence at the Mathematical Sciences Research Institute in Berkeley, California, during the semester of Spring 2019.

The paper was partially prepared during the Simons Semester \emph{Varieties: Arithmetic and Transformations} which is supported by the grant 346300 for IMPAN from the Simons Foundation and the matching 2015-2019 Polish MNiSW fund.

Part of this work was conducted during the semester \emph{Periods in Number Theory, Algebraic Geometry and Physics} at the Hausdorff Institute for Mathematics in Bonn. P.\ A.\ would like to thank the institute for hospitality.

\medskip
\noindent {\bf Note on the use of $\infty$-categories.} For the basics about $\infty$-categories and $\infty$-topoi we refer the reader to the canonical \cite{HTT} (see \cite{groth} for a shorter introduction). The advantage of our functor taking values in the $\infty$-category $\SpSone$ of spaces over the circle and not simply the slice category of the usual homotopy category of spaces is that it allows us to state and prove that this functor is a sheaf (more precisely, a hypercomplete cosheaf), which in turn makes descent arguments possible.  The reader who is uncomfortable with the language can at first ignore all the occurrences of ``$\infty$'' in the text, and just keep in mind that our constructions take values in some homotopy category of spaces.

\medskip
\noindent {\bf Notations.} For a scheme $X$ (or formal scheme, or rigid analytic space) and an open subscheme $U\subseteq X$ with complement $D=X\setminus U$, we denote by $\mathscr{M}_{D}$ the ``compactifying'' log structure on $X$ given by the subsheaf of $\cO_X$ of regular functions that are invertible in $U$, equipped with the inclusion into $\cO_X$. For a log scheme (or formal scheme, or rigid analytic space) $(X,\Mm_X)$ we denote by $(X,\Mm_X)_\triv$ the locus where the log structure is trivial, i.e. where ${\Mm}_{X,x}=\cO_{X,x}^\times$.

We use different scripts to indicate types of geometric objects: $X, Y, \ldots$ for schemes, $\mathfrak{X}, \mathfrak{Y}, \ldots$ for formal schemes, $\mathcal{X}, \mathcal{Y}, \ldots$ for rigid analytic spaces, and $\mathscr{X}, \mathscr{Y}, \ldots$ for models of schemes obtained by a ``spreading out'' procedure.

If $\cO_K$ is a discrete valuation ring whose residue field $k$ is endowed with an embedding ${\iota\colon k \hookrightarrow \CC}$, we often use subscript $(-)_0$ to denote base change of a (formal) scheme over $\cO_K$ along the composition $\cO_K \to k\to \CC$. We denote by $\mathfrak{m}$ the maximal ideal of $\cO_K$, and by $(-)_k$ the  base change along $\cO_K\to k$.

We use the symbols $\mathbf{Sch}$, $\mathbf{Sm}$, $\mathbf{FSch}$ for categories of schemes, smooth schemes, and formal schemes, respectively. The superscripts ${\rm lft, sft, ft, adm, qc, qcs, qcqs}$ mean respectively: locally finite type, separated and finite type, finite type, admissible, quasi-compact, quasi-compact and separated, quasi-compact and quasi-separated.

\tableofcontents

\section{Spreading out}\label{sec:spreading}

The goal of this section is to construct a functor $\Psi$ associating a topological space fibered over $\Sone$ to every scheme locally of finite type over the field $\CCt$. The target of this functor will be the $\infty$-category of spaces over $\Sone$, i.e.\ the topological nerve of the slice category $\mathbf{Top}_{/\Sone}$. In the next section we will give an alternative construction via log geometry, and define $\Psi$ over an arbitrary complete discretely valued field $K$, equipped with an embedding $\iota \colon k\hookrightarrow \CC$ of the residue field $k$ into the field of complex numbers. 

The idea of the construction is the following: a scheme of finite type $X$ over $\CCt$ admits a~model $\Xx/S$ over the spectrum of some smooth finitely generated $\CC[t]$-subalgebra of $\CCs$. The base $S$ comes with a distinguished point $0\colon \Spec \CC\to \Spec \CCs\to S$, and we base change $\Xx/S$ to a small circle $\Sone\to S_\an$ wrapping around the divisor $\{t=0\}\subseteq S$ in a neighborhood of $0\in S$, obtaining $\Psi(X)$. Some work needs to be done to upgrade this to an honest functor to the $\infty$-category of spaces.

\subsection{Models and stratifications}
\label{sec:models}

The following is a special case of N\'eron desingularization (see \cite{Neron}, \cite[\S 4]{ArtinAlgApprox}, or \cite[Tag 0BJ1]{stacks-project}).

\begin{prop}
    Let $f\colon \Lambda\to \cO$ be a homomorphism of discrete valuation rings containing $\mathbf{Q}$. Suppose that $f$ maps a uniformizer in $\Lambda$ to a uniformizer in $\cO$. Then $\cO$ is a filtered colimit of smooth $\Lambda$-algebras of finite type $R$ such that each $R\to \cO$ is injective.
\end{prop}

Applying this to $\Lambda= \CC[t]_{(0)}$ and $\cO = \CCs$, we deduce that $\CCs$ is a filtered colimit of smooth $\CC[t]$-algebras $R$ of finite type, each injecting into $\CCs$. In geometric terms, each $S = \Spec R$ is a complex variety endowed with a smooth map to $\mathbf{A}^1_\CC$ and a formal curve $h\colon \Spec \CCs \to X$ which is ``completely non-algebraic'' in the sense that no element of $B$ vanishes along the image of $h$, and such that $h^*(t) = t\in \CCs$ (see Figure~\ref{fig:SpecR}). We will denote by $0\in S$ the image of $\Spec \CC\to \Spec \CCs\to S$. The same reasoning can be used for $\cO = \CC\{ t\}$, the ring of power series with positive radius of convergence.

\begin{figure}[h]
  \centering
  \includegraphics[width=.5\textwidth]{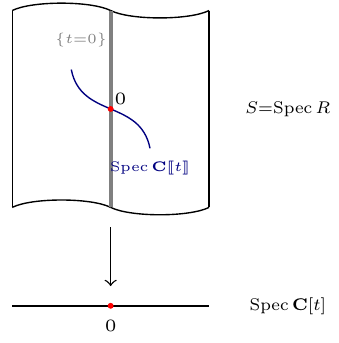} 
  \caption{$S=\Spec R$} 
  \label{fig:SpecR}
\end{figure}

The category of schemes of finite type over $\CCs$ is equivalent to the $2$-categorical colimit of the categories of schemes of finite type over such $R$. By a \emph{model} of a scheme of finite type $X/\CCs$ we shall (until the end of this section) mean a choice of a finite type scheme $\Xx/S=\Spec R$ for some $R$ and an isomorphism $\Xx\otimes_R \CCs \simeq X$.  

It will be useful to choose the models $\Xx/S$ as simple as possible. The first result we shall need to this end is the following.

\begin{thm}[\cite{SMT}] \label{thm:stratif}
  Let $f\colon \Xx\to S$ be a separated morphism between schemes of finite type over $\CC$. Then there exists a finite stratification $|S| = \bigsqcup S_i$ of $|S|$ by connected locally closed subsets $S_i$ such that the restrictions
  \[ 
    f_\an\colon (\Xx_i)_\an \to (S_i)_\an \quad \text{where} \quad \Xx_i =f^{-1}(S_i)
  \]
  are topologically locally trivial fibrations (homeomorphic to a product locally on the base).
\end{thm}

\begin{proof}
By Nagata's theorem, there exists a factorization
\[ 
  \xymatrix{
    \Xx \ar[r]^j \ar[dr]_f & \overline \Xx \ar[d]^{\overline f} \\
    & S
  }
\] 
where $\overline f$ is proper and $j$ is a dense open immersion. Let $Z = \overline \Xx \setminus \Xx$ (with reduced subscheme structure). 

By \cite[1.7]{SMT} (applied to $\overline f$ and $\mathscr{C} = \{Z\}$) there exist finite Whitney stratifications $|S| = \bigsqcup S_i$ and $|\overline \Xx| = \bigsqcup \Xx'_j$ into smooth locally closed connected subsets such that $\overline f$ becomes a stratified map and $Z$ is a union of some of the strata. By definition, $\overline f$ being stratified means that each $\overline\Xx_i := \overline f^{-1}(S_i)$ is a union of some $\Xx'_j$, and that the restrictions $f\colon \Xx'_j\to S_i$ are smooth. 

Since $\overline f$ is proper, Thom's first isotopy lemma \cite[1.5]{SMT} implies that $\overline f_\an\colon ( \overline \Xx_i)_\an \to (S_i)_\an$ are locally trivial fibrations in a stratum-preserving way. Since $\Xx_i$ is the union of some $\Xx'_j$, we conclude that $f_\an\colon (\Xx_i)_\an \to (S_i)_\an$ is a locally trivial fibration as well.
\end{proof}

\begin{remark} \label{rmk:finite-dgm}
  The assertion of Theorem~\ref{thm:stratif} holds more generally for any finite diagram 
  \[
    \mathscr{X}\colon J\to \mathbf{Sch}^{\rm sft}_S
  \]
  of separated schemes of finite type over $S$, where by ``locally trivial fibration'' we mean that locally around each $x\in (S_i)_\an$ there exist $(S_i)_\an$-isomorphisms $\mathscr{X}(j)_{i,\an}
   \simeq (\mathscr{X}(j)_{i,\an})_x \times (S_i)_\an$ such that for every $f\colon j\to j'$ in $J$, the following square commutes
  \[ 
    \xymatrix{
      \mathscr{X}(j)_{i,\an} \ar[r]^-\sim \ar[d]_{\mathscr{X}(f)} & (\mathscr{X}(j)_{i,\an})_x \times (S_i)_\an \ar[d]^{(\mathscr{X}(f)_\an)_x\times {\rm id}} \\
      \mathscr{X}(j')_{i,\an} \ar[r]^-\sim & (\mathscr{X}(j')_{i,\an})_x \times (S_i)_\an.
    }
  \]
  For the proof, we first note that the proof of Theorem~\ref{thm:stratif} easily extends to the case where $\mathscr{X}$ is endowed with a finite stratification. Then we can apply this result to the product $\prod_{j\in J} \mathscr{X}(j)$ with the stratification coming from the graphs of the morphisms $\mathscr{X}(j')\to \mathscr{X}(j)$.
\end{remark}

The following ``stratified version'' of N\'eron desingularization allows one to simplify the stratifications obtained above to $S= \{t\neq 0\} \sqcup \{t=0\}$.

\begin{lemma} \label{lemma:stratifiedneron}
Let $S=\Spec R$ for a smooth $\CC[t]$-subalgebra $R\subseteq \CCs$, and let $|S| = \bigsqcup S_i$ be a finite stratification of $|S|$ by locally closed subsets. Then there exists a smooth $\CC[t]$-subalgebra $R'\subseteq \CCs$ of finite type containing $R$ and such that the image of $S'=\Spec R'\to S$ intersects at most two strata $S_i$ (namely, the open stratum and the one containing $0\in S$).
\end{lemma}

\begin{proof}
Recall that the dilatation (a.k.a.\ N\'eron blowup) of $S$ at a closed subscheme $Z\subseteq \{t = 0\}$ is the affine open subset of the blowup ${\rm Bl}_Z S$ where $t$ generates $I_Z\cdot \cO_{{\rm Bl}_Z S}$, or in other words the complement of the strict transform of $\{t=0\}$ in  ${\rm Bl}_Z S$ (see e.g.\ \cite[Tag~0BJ1]{stacks-project}). In the proof, we can and will often replace $S$ with an open neighborhood of $0\in S$ or with its dilatation at $0$, endowed with the preimage stratification. Note that if $R = \CC[t, x_1, \ldots, x_n] \subseteq \CCs$, then the dilatation at $0=(t, x_1, \ldots, x_n)\in S$ is given by $\Spec R'$ where $R' = \CC[t, x_1/t, \ldots, x_n/t]$. Consequently, $\nu(x_i/t)\geq 0$ where $\nu$ is the $t$-adic valuation, so  $R'\subseteq \CCs$ as well.

After a dilatation at $0\in S$, we may assume that the set $\{t=0\}$ is contained in a~stratum. 

Working one stratum at a time, it is now enough to show the following: if $Z\subseteq \{t\neq 0\}$ is a~proper closed subset, then after replacing $S$ with its dilatation at $0\in S$ finitely many times, the closure of $Z$ does not contain $0\in S$. Let $f$ be a nonzero element of the ideal of $\overline Z$ in $R=\Gamma(S, \cO_S)$ which is not divisible by $t$. We can replace $Z$ with $\{f = 0, t\neq 0\}$, so that $\overline Z = \{f = 0\}$. Let $\nu(f)$ be the valuation of $f$ in $\CCs$. It is enough to show that if $\nu(f)>0$, then after dilatation $S'\to S$ at $0\in S$, the closure $\overline Z'$ of $Z$ in $S'$ is given by an equation $f'=0$ where $\nu(f')<\nu(f)$.

Choose formal local coordinates $t, x_1, \ldots, x_n$ at $0\in S$ such that the $x_i$ lie in the kernel of $\CC[\![t, x_1, \ldots, x_n]\!]=\widehat{\cO}_{S, 0}\to \CCs$. We have an element $f\in \CC[\![t, x_1, \ldots, x_n]\!]$ such that $f\in (t, x_1, \ldots, x_n)$ but $f\notin (x_1, \ldots, x_n)$ and $f\notin (t)$, and $\nu(f)$ is the $t$-adic valuation of $f(t, 0, \ldots, 0)$. The formal completion $\widehat{\cO}_{S',0'}$ of the dilatation $S'$ at the point $0'$ has coordinates $t, y_1, \ldots, y_n$ where $x_i = ty_i$. Since $f$ is in the maximal ideal, the image of $f$ in $\CC[\![t, y_1, \ldots, y_n]\!] = \widehat{\cO}_{S',0'}$ is divisible by $t$, $f = t^k\cdot f'$ where $k>0$ and $f'\notin (t)$. Then the equation of the closure of $\{f=0, t\neq 0\}$ in $S'$ locally at $0'$ is given by $\{f'=0\}$. Therefore
\[ 
  \nu(f') = \nu(f) - k < \nu(f). \qedhere
\]
\end{proof}

Combining Theorem~\ref{thm:stratif} and Lemma~\ref{lemma:stratifiedneron}, we obtain:

\begin{cor} \label{cor:models}
  Let $X$ be a separated scheme of finite type over $\CCs$. Then there exist models $\Xx/S=\Spec R$ over a cofinal system of smooth $\CC[t]$-subalgebras of finite type $R$ of $\CCs$ such that each of
  \[ 
    (\Xx^*)_\an \to (S^*)_\an \quad \text{and} \quad (\Xx_0)_\an \to (S_0)_\an
  \]
  is a locally trivial fibration, where $(-)^*$ (resp.\ $(-)_0$) denotes base change to $\{t\neq 0\}$ (resp.\ $\{t=0\}$). \qed
\end{cor}

By Remark~\ref{rmk:finite-dgm}, the analogous statement holds for any finite diagram of separated schemes of finite type over $\CCs$. 

\subsection{Construction of the functor (I): Separated case}\label{sec:construction}

We shall now define the desired functor $\Psi$ on the full subcategory $\mathbf{Sch}^{\rm sft}_{\CCt} \subseteq \mathbf{Sch}^{\rm lft}_{\CCt}$ consisting of {separated} schemes of finite type over $\CCt$.

We denote by $\Cc$ the partially ordered set consisting of pairs $(S, B)$ where $S= \Spec R$ is the spectrum of a smooth $\CC[t]$-subalgebra of $\CCs$, and where $B \subseteq S_\an$ is an open ball (in some local coordinates) around $0\in S$ such that the map $\pi =  t/|t|\colon B^*\to \Sone$ has contractible fibers. Inequality $(S, B)\leq (S', B')$ means by definition that there exists a map $\pi\colon S'\to S$ under $\Spec \CCs$ (i.e., $S'=\Spec R'$ where $R\subseteq R'$) and $\pi(B') \subseteq B$.

The results of the previous section imply that the poset $\Cc$ is filtering, and that we have an equivalence of categories
\begin{equation} \label{eqn:2colim}
   \colim_{(S, B)\in \Cc}\, \mathbf{Sch}^{{\rm sft}, \star}_S \xrightarrow{\sim} \mathbf{Sch}^{\rm sft}_{\CCs},
\end{equation}
where $\mathbf{Sch}^{{\rm sft}, \star}_S$ is the full subcategory of $\mathbf{Sch}^{{\rm sft}}_S$ consisting of schemes $\Xx/S$ for which the assertion of Corollary~\ref{cor:models} holds. Moreover, base change to $B^*$ defines a functor
\[ 
  (-)_\an\times_{S_\an} B^* \colon \mathbf{Sch}^{{\rm sft}, \star}_S \to {\bf Fib}(B^*)
\]
to the topological category of locally trivial fibrations in CW complexes over $B^*$. 

Pulling back along $\pi\colon B^*\to \Sone$ defines a functor of topological categories
\[ 
  \pi^*\colon {\bf Fib}(\Sone)  \to {\bf Fib}(B^*).
\]
Since $\pi$ is a homotopy equivalence, the induced map of topological nerves
\begin{equation} \label{eqn:fibsone}
  \pi^* \colon N_{\rm top} ({\bf Fib}(\Sone))  \xrightarrow{\sim} N_{\rm top}({\bf Fib}(B^*))
\end{equation}
is an equivalence.

These constructions are functorial with respect to $(S, B)\in \Cc$, and hence define three functors $\Cc\to {\rm\bf Cat}_\infty$ together with two natural transformations
\[
  \xymatrix{
    & & {\bf Sch}^{{\rm sft}, \star}_S \ar[d]^{(-)_\an\times_{S_\an} B^*} \\
  (S, B)\in \Cc \ar@{}[rr]|-\mapsto & & N_{\rm top}({\bf Fib}(B^*)) \\
  & & N_{\rm top}({\bf Fib}(\Sone)) \ar[u]_{\pi^*}
  }
\]
(here the bottom is a constant functor). We claim that $\pi^*  \colon N_{\rm top} ({\bf Fib}(\Sone))  \to N_{\rm top}({\bf Fib}(B^*))$ is in fact an equivalence of functors $\Cc\to {\rm\bf Cat}_\infty$. Indeed, by \cite[5.1.2.1]{HTT}, it is enough to check this on objects, where it follows from \eqref{eqn:fibsone}. Let $\pi_*$ denote an inverse of this equivalence.

Passing to the colimit over $\Cc$ and using the equivalence \eqref{eqn:2colim}, we obtain the desired functor $\Psi$ as the composition
\[ 
  \Psi \colon \mathbf{Sch}^{\rm sft}_{\CCt} \to  \mathbf{Sch}^{\rm sft}_{\CCs} \to \colim_{(S, B)\in \Cc} N_{\rm top}({\bf Fib}(B^*)) \xrightarrow{\pi_*} N_{\rm top}({\bf Fib}(\Sone))
\]
where $\mathbf{Sch}^{\rm sft}_{\CCt} \to  \mathbf{Sch}^{\rm sft}_{\CCs}$ is the functor sending  $X/\CCt$ to $X/\CCs$.

It follows from this definition that for any finite diagram $\{X_i\}_{i\in I}$ in $\mathbf{Sch}^{\rm sft}_{\CCt}$, the value $\Psi(\{X_i\})$ can be calculated by choosing a model $\{\Xx_i\}_{i \in I}$ in ${\bf Sch}^{{\rm sft}, \star}_S$ for some $S$, finding a small ball $B\subseteq S_\an$ containing $0\in S$, and a section $s\colon \Sone\to B^*$ of $\pi\colon B^*\to \Sone$, and then pulling back $\{(\Xx_i)_\an\times_{S_\an} B^*\}_{i\in I}$ along~$s$. 

\subsection{Convergent case}
\label{sec:conv}

Let $\CC(\!\{t\}\!)$ denote the fraction field of the ring $\CC\{t\}$ of power series with positive radius of convergence. If $X$ is a separated scheme of finite type over $\CC(\!\{t\}\!)$, we may spread out $X$ to an analytic space $\Xx/\Delta^*_\varepsilon$ over a punctured disc $\Delta^*_\varepsilon = \{0<|z|<\varepsilon \} \subseteq \CC$, which will be a topological fibration if $\varepsilon$ is small enough. This way, one obtains a functor 
\[ 
  \Psi_{\rm conv}\colon \mathbf{Sch}^{\rm sft}_{\CC(\!\{t\}\!)}   \to N_{\rm top}(\mathbf{Fib}(\Sone)).
\]
We will show below that it agrees with the functor $\Psi$ constructed previously, or more precisely that $\Psi_{\rm conv}(X) \simeq \Psi(X\otimes_{\CC(\!\{t\}\!)} \CCt)$, functorially in $X$.

Let us make the construction of $\Psi_{\rm conv}$ precise. Unfortunately, we do not know of a good notion of a ``scheme of finite type over a punctured disc'' for which there would be an equivalence
\[ 
  \mathbf{Sch}^{\rm sft}_{\CC(\!\{t\}\!)} \isomlong
  \colim_{\varepsilon\to 0} \mathbf{Sch}^{\rm sft}_{\Delta_\varepsilon}.
\]
We therefore use a similar idea to the one used in the definition of $\Psi$. The assertions of \S\ref{sec:models} hold for $\CC\{t\}$ in place of $\CCs$, as they only rely on N\'eron desingularization. We therefore have an equivalence
\[ 
  \colim_{S\in \Cc_0}\, \mathbf{Sch}^{{\rm sft}, \star}_S \isomlong \mathbf{Sch}^{\rm sft}_{\CC\{t\}},
\]
where now $\Cc_0$ is the poset of the spectra $S$ of a smooth $\CC[t]$-subalgebras of $\CC\{t\}$. 

By definition, such an $S$ comes with a morphism $u\colon \Spec \CC\{t\}\to S$, i.e.\ a convergent curve germ. There exists an $\varepsilon > 0$ depending on $S$ such that $u$ defines a holomorphic map 
\[
  u_\an\colon \Delta_\varepsilon \to S_\an.
\]
We choose $\varepsilon(S)$ to be the supremum of such values $\varepsilon$, so that $\varepsilon(S')\leq \varepsilon(S)$ whenever $S'\to S$.  

The pull-back along $u_\an$ defines a functor
\[ 
  u^* \colon \mathbf{Sch}^{{\rm sft}, \star}_S \to \mathbf{Fib}(\Delta_{\varepsilon(S)}^*).
\]
Reasoning as in \S \ref{sec:construction}, we define three functors $\Cc_0\to {\rm\bf Cat}_\infty$ together with two natural transformations
\[
  \xymatrix{
    & & {\bf Sch}^{{\rm sft}, \star}_S \ar[d]^{u^*} \\
  S\in \Cc_0 \ar@{}[rr]|-\mapsto & & N_{\rm top}({\bf Fib}(\Delta^*_{\varepsilon(S)})) \\
  & & N_{\rm top}({\bf Fib}(\Sone)) \ar[u]_{\pi^*}.
  }
\]
As before, we invert the second arrow and pass to the colimit over $\Cc_0$, obtaining the desired
\[ 
  \Psi_{\rm conv} \colon \mathbf{Sch}^{\rm sft}_{\CC(\!\{t\}\!)} \to N_{\rm top}({\bf Fib}(\Sone))
\]
by  composing with the inclusion  $\mathbf{Sch}^{\rm sft}_{\CC(\!\{t\}\!)}\to \mathbf{Sch}^{\rm sft}_{\CC\{t\}}$.

\begin{prop}
  We have a commutative diagram
  \[ 
    \xymatrix{
    \mathbf{Sch}^{\rm sft}_{\CCt} \ar[dr]^\Psi \\ 
    \mathbf{Sch}^{\rm sft}_{\CC(\!\{t\}\!)} \ar[u]^{(-)\otimes_{\CC(\!\{t\}\!)} \CCt} \ar[r]_-{\Psi_{\rm conv}} & N_{\rm top}({\bf Fib}(\Sone)).
    }
  \]
\end{prop}

\begin{proof}
The construction in \S\ref{sec:construction} goes through without change for $\CC\{t\}$ in place of $\CCs$, yielding a functor $\Psi'$ fitting inside a commutative triangle
\[ 
    \xymatrix{
    \mathbf{Sch}^{\rm sft}_{\CCt} \ar[dr]^\Psi \\ 
    \mathbf{Sch}^{\rm sft}_{\CC(\!\{t\}\!)} \ar[u]^{(-)\otimes_{\CC(\!\{t\}\!)} \CCt} \ar[r]_-{\Psi'} & N_{\rm top}({\bf Fib}(\Sone)).
    }
\]
It remains to compare $\Psi'$ with $\Psi_{\rm conv}$. Let $\Cc'$ be the poset of triples $(S, B, \varepsilon)$ with $(S, B)$ as in \S\ref{sec:construction} but with $S\in \Cc_0$, where $0<\varepsilon\leq\varepsilon(S)$ and $B$ contains the image of $\Delta_\varepsilon$. Forgetting $B$ and $\varepsilon$ defines a cofinal map of posets $\Cc'\to \Cc_0$. We have the following commutative diagram of functors from $\Cc'$ to ${\rm\bf Cat}_\infty$:
\[
  \xymatrix@C=0.3em{
    & & & & {\bf Sch}^{{\rm sft}, \star}_S \ar[dll]_{(-)_\an\times_{S_\an} B^*} \ar[d] \ar[drr]^{u^*} \\
  (S, B, \varepsilon)\in \Cc' \ar@{}[rr]|-\mapsto & & N_{\rm top}({\bf Fib}(B^*)) \ar[rr] & &  N_{\rm top}({\bf Fib}(\Delta^*_\varepsilon)) & & N_{\rm top}({\bf Fib}(\Delta^*_{\varepsilon(S)})) \ar[ll] \\
  & & & & N_{\rm top}({\bf Fib}(\Sone)) \ar[ull]^{\pi^*} \ar[urr]_{\pi^*} \ar[u].
  }
\]
Here the left horizontal arrow is the restriction along the induced map $\Delta^*_\varepsilon\to B^*$, which exists thanks to the definition of $\Cc'$. Passing to the limit, and using the fact that $\Cc'$ is cofinal in the poset of all pairs $(S, B)$ used to define $\Psi'$, we obtain the desired natural isomorphism between $\Psi'\to \Psi_{\rm conv}$.
\end{proof}

\subsection{Construction of the functor (II): Descent}\label{sec:descent}

So far, we have only defined the functor $\Psi$ on the category of separated schemes of finite type over $\CCt$. We shall now extend it to all schemes locally of finite type. To this end, it is enough to show that $\Psi$ is a ``homotopy cosheaf''. We refer the reader to \cite[Section 6.5.3]{HTT} for basics on hypercoverings.

\begin{defin} \label{def:hypercosheaf}
Let $\Cc$ be a site and $\mathscr{D}$ be an $\infty$-category. We say that a functor $F\colon \Cc\to \mathscr{D}$ is a~\emph{hypercosheaf} if for every hypercovering $Y_\bullet\to X$ in $\Cc$ the natural map $\hocolim F(Y_\bullet)\to F(X)$ is an equivalence in $\mathscr{D}$.
\end{defin}

We denote by $\mathbf{Hyp}^{\rm co}_\infty(\Cc,\mathscr{D})$ the $\infty$-category of $\mathscr{D}$-valued hypercosheaves $F\colon  \Cc\to \mathscr{D}$. 

In order to extend $\Psi$ we will prove that it is a hypercosheaf on ${\bf Sch}^{\rm sft}_{\CCt}$, and then we will make use of the following general lemma to argue that it will uniquely extend to a hypercosheaf on the whole ${\bf Sch}^{\rm lft}_{\CCt}$.

\begin{prop}\label{prop:extend-cosheaf}
Let $\Cc$ and $\Cc'$ be sites, and $F\colon \Cc\to \Cc'$ be a continuous functor inducing an equivalence of topoi ${\bf Sh}(\Cc')\to{\bf Sh}(\Cc)$. Then for every $\infty$-category $\mathscr{D}$ we have a natural equivalence of $\infty$-topoi $\mathbf{Hyp}^{\rm co}_\infty(\Cc',\mathscr{D})\to \mathbf{Hyp}^{\rm co}_\infty(\Cc, \mathscr{D})$.
\end{prop}

\begin{proof}

Note first of all that it suffices to prove the analogous statement for the $\infty$-topos $\textbf{Hyp}_\infty(-)$ of hypercomplete sheaves (i.e. hypersheaves, the obvious variant of Definition \ref{def:hypercosheaf}) of spaces, because the $\infty$-topos $\mathbf{Hyp}^{\rm co}_\infty(-, \mathscr{D})$ can be identified with the $\infty$-category of colimit-preserving functors $\mathbf{Hyp}_\infty(-)\to \mathscr{D}$ (by the same reasoning of \cite[Proposition 1.1.12]{dag5}).

The statement for $\textbf{Hyp}_\infty(-)$ follows immediately from the fact that the Jardine model structure on simplicial sheaves \cite{jardine} presents the $\infty$-topos of hypercomplete sheaves \cite[Proposition 6.5.2.14]{HTT}.
\end{proof}

Recall that the $h$-topology \cite[\S 10]{SuslinVoevodsky} is the topology generated by universal topological epimorphisms, or equivalently by proper surjections and Zariski coverings.

\begin{prop}\label{prop:descent}
The functor $\Psi\colon {\bf Sch}^{\rm sft}_{\CCt}\to \SpSone$ is a hypercosheaf for the $h$-topology.
In other words, if $Y_\bullet\to X$ is an $h$-hypecovering in ${\bf Sch}^{\rm sft}_{\CCt}$, then the induced map
  \[ 
    \hocolim \Psi(Y_\bullet) \to \Psi(X)
  \]
  is an equivalence.
\end{prop}

\begin{proof}
The proof will be in a few steps. We will first reduce the  statement to the case of the \v{C}ech nerve of a single $h$-covering $Y\to X$, and then combine results of Dugger--Isaksen and Blanc on hyperdescent for topological spaces in the classical and proper topologies to conclude.

\textbf{Step (1)}: it suffices to  prove the statement for bounded hypercoverings $Y_\bullet\to X$.

Recall that a hypercovering  is bounded if it is of the form $\cosk_n^X(Y_\bullet)\to X$ for some $n\geq 0$ (i.e.  if the unit map $Y_\bullet\to \cosk_n^X(Y_\bullet)$ is an isomorphism --- the minimum $n$ for which this happens is called the \emph{dimension} of $Y_\bullet$). Fix $k\geq 0$, and consider  the coskeleton $\cosk_{k+1}^X(Y_\bullet)$. There is a~diagram
\[
\xymatrix{
\hocolim \Psi(Y_\bullet)\ar[r] \ar[rd] & \hocolim\Psi(\cosk_{k+1}^{X}(Y_\bullet))\ar[d]\\
  & \Psi(X).
}
\]
Note that since $\Psi$ is a functor only in the $\infty$-categorical sense, $\Psi(Y_\bullet)$ is not an honest simplicial space. If it were, and if $\Psi(\cosk_{k+1}^{X}(Y_\bullet)) = \cosk_{k+1}^{\Psi(X)} (\Psi(Y_\bullet))$, then we could argue as follows: since $\Psi(Y_\bullet)\to\Psi(\cosk_{k+1}^{X}(Y_\bullet))$ induces an isomorphism on the $(k+1)$-skeleton,  by \cite[Lemma 3.25]{Blanc} the horizontal map induces an isomorphism on $\pi_k$ at any basepoint. If we assume the statement for bounded hypercoverings, then $\hocolim\Psi(\cosk_{k+1}^{X}(Y_\bullet))\to \Psi(X)$ is an equivalence, and it follows that $\hocolim \Psi(Y_\bullet)\to \Psi(X)$ induces an isomorphism on $\pi_k$ at any basepoint. Therefore this map will be an equivalence, since $k$ is arbitrary.

The remainder of the proof of Step (1) will consist of ``straightening'' the functor $\Psi$ for simplicial objects, i.e.\ building an augmented simplicial space $\overline Y_\bullet \to \overline X$ modelling $\Psi(Y_\bullet)\to \Psi(X)$, and such that there is a natural identification $\cosk_{k+1}^{\overline X}(\overline Y_\bullet)  \simeq \Psi(\cosk_{k+1}^{X}(Y_\bullet))$. The problem here being that since $Y_\bullet$ is an infinite diagram of spaces, it might not admit a model over some $S = \Spec R$ satisfying suitable conditions. In any case, for finite truncations of the simplicial scheme $Y_\bullet$, there exist such models $\mathscr{Y}_{\bullet\leq n} \to \mathscr{X} \times_S S_n$ over $(S_n, B_n)$, where $\mathscr{X}/(S, B)$ is some fixed model of $X$, which are moreover locally constant as finite diagrams of spaces over $B_n^*$ in the sense of Remark~\ref{rmk:finite-dgm}. We can build them iteratively, i.e.\ such that there exist maps $(S_{n+1}, B_{n+1}) \to (S_n, B_n)$ and compatible isomorphisms
\[ 
  (\mathscr{Y}_{\bullet \leq n+1})_{\leq n} \simeq (\mathscr{Y}_{\bullet\leq n}) \times_{S_n} S_{n+1}.
\] 
Since $B_n^* \to \Sone$ is a homotopy equivalence and $(\mathscr{Y}_{\bullet\leq n})_\an \times_{(S_n)_\an} B_n^*$ is a locally constant finite diagram of spaces, it is the pull-back of a truncated simplicial space $\overline Y_{\bullet\leq n}$ fibered over $\Sone$. By construction, we have homeomorphisms $(\overline Y_{\bullet\leq n+1})_{\leq n} \simeq \overline Y_{\bullet\leq n}$  of truncated simplicial spaces fibered over $\Sone$. We therefore obtain an augmented simplicial space $\overline Y_\bullet \to \overline X$ fibered over $\Sone$ which is the required model. The compatibility with coskeleta follows from the fact that $(-)_\an \times_{S_\an} B^*$ commutes with fiber products.

\textbf{Step (2)}: it suffices to prove the statement for the \v{C}ech nerve $C(Y/X)_\bullet\to X$ of a single $h$-covering $Y\to X$.

This can be proven exactly as the analogous fact in the proof of \cite[Proposition 3.24]{Blanc}, once we observe that the simplicial object $\Psi(\cosk_{k+1}^{X}(Y_\bullet))$ can be lifted to a~simplicial object of dimension $k+1$ in the category $\bf Fib(\Sone)$. This follows from the same arguments used in the construction of $\overline Y_\bullet \to \overline X$ in the previous step. 

\textbf{Step (3)}: we now prove the statement for the \v{C}ech nerve $C(Y/X)_\bullet\to X$ of a single $h$-covering $Y\to X$.

First of all we argue that the result is true if $Y\to X$ is either a Zariski covering or a proper surjection. In the first case, this follows from the fact that $\Psi(C(Y/X)_\bullet)\to \Psi(X)$ is a hypercovering of topological spaces for the classical topology, and then what we are after is precisely \cite[Theorem~1.3]{DuggerIsaksen}. Note that by the same argument as above, $\Psi(C(Y/X)_\bullet)$ can in fact be lifted to a~simplicial object of dimension $0$ in $\bf Fib(\Sone)$, as the \v{C}ech nerve of a model $\mathscr{Y}\to \mathscr{X}$ for $Y\to X$ over some $S=\Spec R$, pulled back to $B^*\subseteq B\subseteq S_\an$ as usual. The argument in the second case is exactly the same, using \cite[Proposition~3.24]{Blanc} (note that the proper hypercovering that we obtain from a model $\mathscr{Y}\to \mathscr{X}$ satisfies the required ``niceness'' assumptions).

Now for the general case, observe that the $h$-topology of ${\bf Sch}^{\rm sft}_{\CCt}$ is generated by Zariski coverings and proper surjections. From this it follows that a functor $F\colon {\bf Sch}^{\rm sft}_{\CCt}\to \mathscr{D}$ to some $\infty$-category $\mathscr{D}$ is a cosheaf for the $h$-topology if and only if it is both a cosheaf for the Zariski topology and for the proper topology. Because of what we just proved, $\Psi\colon {\bf Sch}^{\rm sft}_{\CCt}\to \SpSone$ is a cosheaf for the $h$-topology. This exactly means that for an $h$-covering $Y\to X$, the map $\hocolim \Psi(C(Y/X)_\bullet)\to \Psi(X)$ is an equivalence.
\end{proof}

\begin{cor}
There is a unique extension of the functor $\Psi\colon \mathbf{Sch}_{\CCt}^{\rm sft}\to \SpSone$ to a hypercosheaf $\Psi\colon  \mathbf{Sch}^{\rm lft}_{\CCt}\to \SpSone$ with respect to the $h$-topology.
\end{cor}

\begin{proof}
This follows immediately from Propositions \ref{prop:extend-cosheaf} and \ref{prop:descent}.
\end{proof}

We can easily show that in fact the functor descends to the Morel--Voevodsky $\mathbf{A}^1$-homotopy category. We will not use this observation in the sequel.

\begin{cor}
Let ${\bf MV}_S$ denote the $\AA^1$-homotopy category of Morel--Voevodsky over $S$ \cite{morel-voevodsky}. % and $\mathcal{S}$ be the $\infty$-category of spaces. 
Then the functor $\Psi$ induces a functor ${\bf MV}_{\CCt}\to \SpSone$, which is compatible with the usual Betti realization ${\bf MV}_{\CC}\to \Spaces$, i.e. the diagram
\[
\xymatrix{
{\bf MV}_{\CC}\ar[r]\ar[d]_{(-) \times_\CC \CCt} & \Spaces \ar[d]^{(-)\times \Sone}\\
{\bf MV}_{\CCt}\ar[r] & \SpSone
}
\]
commutes.
\end{cor}

\begin{proof}
By the results of this section the functor $\Psi$ satisfies Nisnevich descent. Moreover, it is clear that it carries projections $V\times \AA^1\to V$ to equivalences, and that if $V$ is a smooth scheme over $\CC$, then $\Psi(V\times_\CC \CCt)\simeq ( V_\an\times \Sone\xrightarrow{\pi_2} \Sone)$.
\end{proof}

\subsection{\texorpdfstring{Construction of the functor (III): Arbitrary $K$}{Construction of the functor (III): Arbitrary K}} \label{sec:arbitrary.K}

We can apply the construction outlined here over any complete discretely valued field $K$ equipped with an embedding $\iota \colon k\hookrightarrow \CC$, after choosing a compatible continuous embedding $K\hookrightarrow \CCt$. We will prove later in \S\ref{sec:log.geometry} that the construction is independent of this choice.

\begin{lemma}
Let $K$ be a complete discretely valued field with an embedding $\iota \colon k\hookrightarrow \CC$ of its residue field into the field of complex numbers. Then there exists a compatible continuous embedding $\widetilde\iota\colon K\hookrightarrow \CCt$.
\end{lemma}

\begin{proof}
By Cohen's structure theorem \cite[\href{https://stacks.math.columbia.edu/tag/0323}{Tag 0323}]{stacks-project} we have a non-canonical isomorphism $\cO_K\simeq k[\![t]\!]$. In turn, this induces an isomorphism $K\simeq k(\!(t)\!)$ and a continuous embedding $
K\simeq k(\!(t)\!)\hookrightarrow \CCt
$
as desired.
\end{proof}

Over any such field $K$ we can now define $\Psi\colon {\bf Sch}^{\rm lft}_K\to \SpSone$ as the composite
\[
\Psi\colon {\bf Sch}^{\rm lft}_K\xlongrightarrow{\widetilde\iota^*} {\bf Sch}^{\rm lft}_\CCt\to \SpSone
\]
where $\widetilde\iota^*$ is the base change functor induced by the chosen embedding $\widetilde\iota\colon K\hookrightarrow \CCt$, and the second arrow is the functor $\Psi$ for $\CCt$ constructed previously.

\section{Log geometry}\label{sec:log.geometry}

\subsection{Review of Kato--Nakayama spaces}
\label{sec:knreview}

In this section, we briefly recall the construction and properties of Kato--Nakayama spaces that will be used in the rest of the paper. For a good introduction to log geometry we refer the reader to the survey \cite{abramovich} or the book \cite{ogus}. For more details on Kato--Nakayama spaces, see \cite{KN} or \cite[Section V.1]{ogus}.

Let $(X,\Mm_X)$ be a fine log complex analytic space (that later on will always be the analytification of a log scheme locally of finite type over $\CC$). The Kato--Nakayama space is a topological space $(X,\Mm_X)_\log$, equipped with a proper map $\tau\colon (X,\Mm_X)_\log\to X$, whose topology reflects the log geometry of $(X, \Mm_X)$. 

As a set, $(X,\Mm_X)_\log$ is defined as the set of pairs $(x,\phi)$, where $x\in X$ and $\phi$ is a homomorphism of abelian groups $\phi\colon \Mm_{X,x}^\gp\to \Sone$, such that $\phi(f)={f(x)}/{|f(x)|}\in \Sone$ for every invertible section $f\in \cO_{X,x}^\times$. The projection $\tau\colon (X,\Mm_X)_\log\to X$ is the projection to the first coordinate $(x,\phi)\mapsto x$.
For every open subset $U\subseteq X$ and section $m\in \Mm_X(U)$, we obtain a function $f_{U,m}\colon U_\log\to \Sone$, defined by $(x,\phi)\mapsto \phi(m_x)\in \Sone$. Here $U_\log$ denotes the subset of $(X,\Mm_X)_\log$ of points $(x,\phi)$ with $x\in U$, and can be identified with the Kato--Nakayama space of the log analytic space $(U,\Mm_X|_U)$. The topology on $(X,\Mm_X)_\log$ is defined as the coarsest topology that makes the projection $\tau\colon (X,\Mm_X)_\log\to X$ and all the functions $f_{U,m}\colon U_\log\to S^1$ continuous.

The map $\tau\colon (X,\Mm_X)_\log\to X$ is a proper continuous map, and can be seen as a relative compactification of the open embedding $(X,\Mm_X)_\triv\subseteq X$, since this factors through an open embedding $(X,\Mm_X)_\triv\subseteq (X,\Mm_X)_\log$. The fiber over a point $x\in X$ is a torsor under the space of homomorphisms of abelian groups $\Hom(\Mm_{X,x}^\gp,\Sone)$, and if the log analytic space $X$ is fs (or more generally if $\Mm_{X,x}^\gp$ is torsion-free), this is non-canonically isomorphic to a real torus $(\Sone)^r$ for some $r$.

The formation of $(X,\Mm_X)_\log$ is functorial and compatible with base change with respect to strict morphisms.

\begin{example}\label{example:kn}\mbox{}
\begin{itemize}
\item[a)] Let $(X,\Mm_X)$ be $\AA^1$, with its toric log structure. Then $(X,\Mm_X)_\log=\overline{\CC}:=\RR_{\geq 0}\times \Sone$, and the projection $\tau\colon (X,\Mm_X)_\log\to X$ is given by the map $\overline{\CC}=\RR_{\geq 0}\times \Sone\to \CC$ sending $(r,a)$ to $r\cdot a$ (where we see $\Sone\subseteq \CC$ as the complex numbers of norm $1$). \item[b)] Generalizing the previous example, let $P$ be a fine monoid, and let $(X,\Mm_X)$ be the affine toric scheme $\Spec(P\to \CC[P])$ with its toric log structure. Then $(X,\Mm_X)_\log=\Hom(P, \overline{\CC})$, with its natural topology induced by the topology on $\overline{\CC}$. 
\item[c)] If $X$ is smooth, and $\Mm_D$ is the compactifying log structure coming from a simple normal crossings divisor $D\subseteq X$, then $(X,\Mm_D)_\log$ can be identified with the real oriented blowup of $X$ along $D$, and is a ``smooth manifold with corners''.
\end{itemize}
\end{example}

The last example suggests that $(X,\Mm_X)_\log$ should be thought of as the complement of an ``open tubular neighbourhood of the log structure.''

\begin{prop} \label{prop:triv-log-equiv}
Let $X$ be a smooth complex analytic space and $D\subseteq X$ be a simple  normal crossings divisor. Then the inclusion $X\setminus D \to (X, \Mm_D)_{\rm log}$ is a homotopy equivalence. 
\end{prop}

\begin{proof}
This follows from the fact that  $(X,\Mm_D)_\log$ is a topological manifold with boundary, whose interior is exactly $X\setminus D$ \cite[Theorem V.1.3.1]{ogus}.
\end{proof}

For future reference, we also record some topological properties of the Kato--Nakayama space of a log analytic space in the following proposition.

\begin{prop}\label{prop:top.properties}
Let $(X,\Mm_X), (Y,\Mm_Y)$ be separated, fs, separable log analytic spaces, and $f\colon (X,\Mm_X)\to (Y,\Mm_Y)$ be a morphism. Then:
\begin{itemize}
\item[a)] the Kato--Nakayama space $(X,\Mm_X)_\log$ is locally compact, locally contractible separable metric space, and 
\item[b)] for every $y\in (Y,\Mm_Y)_\log$, the fiber $f_\log^{-1}(y)\subseteq (X,\Mm_X)_\log$ is locally contractible.
\end{itemize}
\end{prop}

\begin{proof}
For point a), all the claims follows from the fact that with our assumptions, the spaces $(X,\Mm_X)_\log$ and $(Y,\Mm_Y)_\log$ are locally triangulable and Hausdorff. See \cite[Proposition 5.3]{nakayama-ogus} or \cite[Proposition A.13]{knvsroot} for details. As for part b), one can use similar arguments by noting that the map $f_\log$ is semi-analytic, and hence its fibers are semi-analytic.
\end{proof}

To conclude, let us consider the effect of the functor $(-)_\log$ on log smooth degenerations $f\colon (X,\Mm_X)\to (S,\Mm_S)$. In this setup the induced map $f_\log\colon (X,\Mm_X)_\log\to (S,\Mm_S)_\log$ (and especially the fibers over the critical values of $f$) gives a geometric model for nearby/vanishing cycles and monodromy of the degeneration. 

This interpretation hinges on the following theorem of Nakayama and Ogus \cite{nakayama-ogus}. Recall that a homomorphism of monoids $P\to Q$ is \emph{exact} if the natural map $P\to Q\times_{Q^\gp} P^\gp$ is an isomorphism, and a morphism of  log schemes $f\colon (X,\Mm_X)\to (S,\Mm_S)$  is \emph{exact} if for every point $x\in X$  the homomorphism $(\overline{\Mm_S})_{f(x)}\to ( \overline{\Mm_X})_{x}$ is  exact.

\begin{thm}[{\cite[Theorems 3.5 and 5.1]{nakayama-ogus}}]\label{thm:nakayama-ogus}
Let $f\colon (X,\Mm_X)\to (S,\Mm_S)$ be an exact and log smooth morphism of fine log analytic spaces. Then the induced map of Kato--Nakayama spaces $f_\log\colon (X,\Mm_X)_\log\to (S,\Mm_S)_\log$ is a topological submersion, i.e.\ locally on $(X,\Mm_X)_\log$ and $(S,\Mm_S)_\log$ it is isomorphic to a projection $Z\times (S,\Mm_S)_\log\to (S,\Mm_S)_\log$ for some space $Z$. 

If moreover $f$ is proper, then $f_\log$ is a topological fiber bundle, i.e.\ locally on $(S,\Mm_S)_\log$ it can be identified with the projection $Z\times (S,\Mm_S)_\log\to (S,\Mm_S)_\log$ for some space $Z$.
\end{thm}

We note that a log smooth morphism $f\colon (X,\Mm_X)\to (S,\Mm_S)$ as above, where the stalks of $\overline\Mm_S$ are either $\NN$ or $0$ (e.g.\ the standard log point), is automatically exact.

\subsection{Good models}

In this section we define good  models over $\Spec \cO_K$ of smooth schemes over a complete discretely valued field $K$. We denote by $\mathfrak{m}\in \Spec \cO_K$ the closed point, and consider also the standard log structure $ \Mm_{\mathfrak{m}}$ on $\Spec \cO_K$, given by the divisor $\{\mathfrak{m}\}\subseteq \Spec \cO_K$.

\begin{defin}\label{def:good.model}
  A \emph{good model} is a pair $\XX = (\mathscr X, \mathscr D)$ consisting of a proper flat regular $\cO_K$-scheme $\mathscr X$ and a divisor $\mathscr D\subseteq \mathscr X$ with simple normal crossings such that $(\mathscr X_k)_{\rm red}\subseteq \mathscr D$.
  
  A morphism of good models $\pi\colon \XX'=(\mathscr X', \mathscr D')\to (\mathscr X, \mathscr D)=\XX$ is a map $\pi\colon \mathscr X'\to \mathscr X$ of $\cO_K$-schemes such that $\pi^{-1}(\mathscr D)\subseteq \mathscr D'$, or equivalently $\pi(U')\subseteq U$ where $U=\mathscr X\setminus \mathscr D$, $U'=\mathscr X'\setminus \mathscr D'$. This defines the category $\mathbf{GM}_{\cO_K}$ of good models over $\cO_K$.
  
  If $\XX=(\mathscr X,\mathscr D)$ is a good model, its \emph{associated log scheme} is the log scheme $(\mathscr X,  \Mm_{\mathscr D})$. 
\end{defin}

\begin{remark}
Every separated smooth scheme of finite type over $K$ admits a good model over $\cO_K$, by appropriate versions of Nagata compactification \cite{conrad-nagata} and Hironaka's resolution of singularities \cite{temkin}.
\end{remark}

\begin{prop}\label{prop:gmlogsmooth}
Let $\XX=(\mathscr X,\mathscr D)\in \mathbf{GM}_{\cO_K}$.
\begin{enumerate}[(a)]
  \item The log scheme $(\mathscr X, \Mm_{\mathscr D})$ is log smooth and exact over $(\Spec \cO_K, \Mm_\mathfrak{m})$. 
  \item A morphism $f \colon  \XX'\to \XX$ in $\mathbf{GM}_{\cO_K}$ extends uniquely to a morphism 
  \[ 
    f\colon  (\mathscr X', \Mm_{\mathscr D'})\to (\mathscr X, \Mm_{\mathscr D}).
  \]
  \item This defines a fully faithful functor 
  \[ 
      \XX \mapsto (\mathscr X, \Mm_{\mathscr D}) 
      \quad \colon  \quad
      \mathbf{GM}_{\cO_K} \to \mathbf{LS}_{/\cO_K}
  \]
  where $\mathbf{LS}_{/\cO_K}$ denotes the category of fs log schemes equipped with a log smooth (and exact) map to $(\Spec \cO_K, \Mm_\mathfrak{m})$.
\end{enumerate}
\end{prop}

\begin{proof}
To prove (a), choose a point $x\in \mathscr X$, and let $s$ be the image in $\Spec \cO_K$. Then if $s$ is the generic point $\eta$, Zariski locally around $x$ the map $(\mathscr X,\mathscr D)\to \Spec \cO_K$ is described by $(\mathscr X_\eta,\mathscr D_\eta)\to \eta$, which is log smooth, since $\mathscr D_\eta\subseteq \mathscr X_\eta$ is simple normal crossings and $\eta$ has the trivial log structure. Otherwise, the image of $x$ is the closed point $\mathfrak{m}\in \Spec \cO_K$, and locally around $x$ we can find a commutative diagram
\begin{equation}\label{diagram}
\begin{aligned}
\xymatrix{
\mathscr X\ar[r]\ar[d] & \AA^n\ar[d] \\
\Spec \cO_K\ar[r] & \AA^1
}
\end{aligned}
\end{equation}
with strict horizontal arrows, where: $\AA^n\to \AA^1$ is determined by a non-trivial homomorphism of monoids $\NN\to \NN^r\to \NN^n$ (where the second map is inclusion as $\NN^r\times \{0\}\subseteq \NN^r\times \NN^{n-r}$), the map $S=\Spec \cO_K \to \AA^1$ is given by the choice of a uniformizer, and the map $\mathscr X\to \AA^n$ is determined by local parameters at $x$ (the first $r$ of which correspond to branches of $\mathscr D$).

Let $\mathscr X'=\Spec \cO_K\times_{\AA^1} \AA^n$ be the fibered product of the diagram, that can be explicitly described as $\Spec \cO_K [x_1,\hdots, x_n]/(u-\prod_{i=1}^r x_i^{a_i})$, where $u$ is a uniformizer of $\cO_K$ and $(a_1,\hdots, a_r)$ is the image of $1$ via the map $\NN\to \NN^r$. We have to check that the induced morphism $\mathscr X\to \mathscr X'$ is smooth at $x$. Note that it will actually be \'etale, since the relative dimension is $0$.

It suffices to check that the map is formally smooth at $x$, and we can do that by looking at the induced morphism $\Spec \widehat{\cO}_{\mathscr X,x}\to \mathscr X'$. 
This map is described by a homomorphism
\[
 \cO_K[x_1,\hdots, x_n]/(u-\prod_{i=1}^r x_i^{a_i})\to \widehat{\cO}_{\mathscr X,x}.
\]
sending the $x_i$ to local parameters at $x$. It follows from the infinitesimal lifting criterion that this morphism is formally smooth.
This proves that $(\mathscr X,\Mm_{\mathscr D})\to (\Spec \cO_K,\Mm_\mathfrak{m})$ is log smooth.

As for exactness, if $x$ maps to the generic point $\eta\in \Spec \cO_K$ there is nothing to prove (the morphism of monoids $\{0\}\to P$ is exact if $P$ is integral). Otherwise, the morphism $(\overline{\Mm_\mathfrak{m}})_{\mathfrak{m}}\to (\overline{\Mm_{\mathscr D}})_{x}$ is a non-trivial homomorphism of monoids $\NN\to \NN^r$, and it is immediate to check that all such morphisms are exact.

Part (b) is obvious, and part (c) follows from the general fact that if a scheme $X$ has the compactifying log structure with respect to an open subscheme $U\subseteq  X$ with complement $D$, then for every log scheme $(Y,\Mm_Y)$, morphisms of log schemes $( X,\Mm_{ D})\to (Y,\Mm_Y)$ correspond bijectively to morphisms of schemes $f\colon X\to Y$ such that $f(U)\subseteq (Y,\Mm_Y)_\triv$ \cite[Proposition III.1.6.2]{ogus}. 
\end{proof}

The previous proof suggests the following definition of a good model over a smooth scheme $S$ over $\Spec \cO_K$.

\begin{defin}\label{def:good.models.S}
Let $S$ be a smooth scheme over $\Spec \cO_K$, and denote by $S_k\subseteq S$ the fiber over $\Spec k$.

A \emph{good model over} $S$ is a pair $\XX = (\mathscr X, \mathscr D)$ consisting of a proper flat regular $S$-scheme $\mathscr X$ and a divisor $\mathscr D\subseteq \mathscr X$ with simple normal crossings, such that \'etale locally around points of the fiber $\mathscr X_k$ of $\mathscr X$ over $\Spec k$, the map $ \mathscr X\to S$ admits a chart, i.e. a diagram (\ref{diagram}) as in the previous proof, where $S\to \AA^1$ is given by a local equation of the subscheme $S_k\subseteq S$ and the map to the fibered product $\mathscr X\to S\times_{\AA^1} \AA^n$ is \'etale.

A morphism of good models $\pi\colon \XX'=(\mathscr X', \mathscr D')\to (\mathscr X, \mathscr D)=\XX$ over $S$ is a map $\pi\colon \mathscr X'\to \mathscr X$ of $S$-schemes such that $\pi^{-1}(\mathscr D)\subseteq \mathscr D'$, or equivalently $\pi(U')\subseteq U$ where $U=\mathscr X\setminus \mathscr D$, $U'=\mathscr X'\setminus \mathscr D'$. This defines the category $\mathbf{GM}_{S}$ of good models over $S$.
 
If $\XX=(\mathscr X,\mathscr D)$ is a good model, its \emph{associated log scheme} is the log scheme $(\mathscr X,  \Mm_{\mathscr D})$ over $(S,\Mm_{S_k})$.
\end{defin}

The analogue of Proposition \ref{prop:gmlogsmooth} holds for good models over a smooth $\cO_K$-scheme $S$ as well, with a similar proof.

From Proposition \ref{prop:gmlogsmooth} and Theorem \ref{thm:nakayama-ogus}, we obtain the following.

\begin{cor}
Let $\XX=(\mathscr X, \mathscr D)$ be a good model over $\cO_K$, and consider the induced morphism of log schemes $(\mathscr X,\Mm_{\mathscr D})_k\to \Spec (\NN\to k)$ from the central fiber to the standard log point over $k$. Consider moreover the base change $(\mathscr X,\Mm_{\mathscr D})_{0}\to 0=\Spec(\NN\to \CC)$ of this morphism along the embedding $\iota\colon k\hookrightarrow \CC$. Then the induced map 
\[
  (\mathscr X,\Mm_{\mathscr D})_{0,{\log}}\to 0_\log=\Spec(\NN\rightarrow \CC)_\log=\Sone
\]
is a fiber bundle. \qed
\end{cor}

\subsection{Construction of the functor (IV): Comparison}
\label{ss:comparison}

Denote by $\Psi_{\rm log}$ the following composition
\[ 
  \Psi_{\rm log} \colon \mathbf{GM}_{\cO_K}\to \mathbf{LS}_{/\cO_K} \to \mathbf{LS}_{/0} \to {\bf Fib}(\Sone),
\]
\[
  (\mathscr X, \mathscr D) \mapsto (\mathscr X, \Mm_{\mathscr D})_{0, \rm log} / \Sone
\]
where, as in the previous corollary, $(\mathscr X, \Mm_{ \mathscr D})_{0}$ denotes the base change of the central fiber $(\mathscr X,\Mm_{\mathscr D})_k$ along the embedding $\iota\colon k\hookrightarrow \CC$.

\begin{prop} \label{prop:comparison}
  Suppose that $K=\CCt$. Then the following triangle of functors naturally commutes
  \[ 
    \xymatrix{
      \mathbf{GM}_{\CC[\![t]\!]} \ar[dr]^{\Psi_{\rm log}} \ar[d]_{(-)_{\rm triv}}  \\
      \mathbf{Sch}^{\rm sft}_\CCt \ar[r]_-{\Psi} & \SpSone.
    }
  \]
\end{prop}

\begin{proof}
Similarly to $\mathbf{Sch}^{\rm sft}_\CCt$, we can realize $\mathbf{GM}_{{\CC[\![t]\!]}}$ as the colimit of categories $\mathbf{GM}_S$ of good models defined over $S=\Spec R$ for $R$ a smooth $\CC[t]$-subalgebra of $\CCs$, in the sense of Definition \ref{def:good.models.S}. %where we equip $S$ with the log structure induced by the open subset $S^* = \{t\neq 0\}$. 

If $B$ is an open subset of $S_\an$, we endow $B$ with the induced log structure, so that we have the topological space $B_{\rm log}$ (in fact, $B_{\rm log} = S_{\rm log}\times_{S_\an} B$). Let $\Cc$ again denote the category of pairs $(S, B)$ as in \S\ref{sec:construction}. Then the maps in the diagram
\[ 
  \xymatrix{
    B^* \ar[dr]_{\pi} \ar[r]^j & B_{\rm log} & 0_{\rm log} \ar[l]_i \ar@{=}[dl] \\
    & \Sone 
  }
\]
are all equivalences.

We have the following commutative diagram of functors from $\Cc$ to ${\rm\bf Cat}_\infty$:
\[
  \xymatrix{
    & & \mathbf{GM}_S \ar[r]^-{(-)_{\rm triv}} \ar[d]_{(-)_{\rm log}\times_{S_{\rm log}} B_{\rm log}} & {\bf Sch}^{{\rm sft}, \star}_S \ar[d]^{(-)_\an\times_{S_\an} B^*} \\
  (S, B)\in \Cc \ar@{}[rr]|-\mapsto & & N_{\rm top}({\bf Fib}(B_{\rm log})) \ar[d]_{i^*} \ar[r]^{j^*} &  N_{\rm top}({\bf Fib}(B^*)) \\
  & & N_{\rm top}({\bf Fib}(0_\log)) \ar@{=}[r] & N_{\rm top}({\bf Fib}(\Sone)) \ar[u]_{\pi^*}.
  }
\]
Here, the nontrivial statement is that the upper square commutes. In fact, the square comes with a natural transformation between the two compositions, the base change to $B^*$ of the functorial inclusion
\[ 
 (\mathscr X,\Mm_{\mathscr D})_\triv= \mathscr X\setminus \mathscr D \hookrightarrow (\mathscr X, \Mm_{\mathscr D})_{\rm log},
\]
which is an equivalence thanks to Proposition~\ref{prop:triv-log-equiv}.

Passing to the colimit over $(S, B)\in \Cc$ and inverting the equivalences coming from the bottom square, we obtain the desired diagram.
\end{proof}

This comparison result allows us to deduce nice consequences for both versions of the construction.

\begin{prop} Let $K$ be a complete discretely valued field whose residue field $k$ is equipped with an embedding $\iota\colon k\hookrightarrow \CC$.
\begin{enumerate}[(a)]
\item The object $\Psi_\log(\XX)$ only depends on  the open subscheme $\mathscr X\setminus \mathscr D$ (which is a $K$-scheme), and it is well-defined up to homeomorphism. 
\item The resulting functor $\Psi_\log\colon {\bf Sm}^{\rm sft}_K\to \SpSone$ from the category of separated, smooth, finite type schemes over $K$ extends uniquely to a functor on ${\bf Sch}^{\rm lft}_K$, which is a hypercosheaf (see \S\ref{sec:descent}) for the $h$-topology.
\item For $X\in {\bf Sch}^{\rm lft}_K$, the object $\Psi(X)$ defined in \S\ref{sec:arbitrary.K} only depends on the embedding $\iota\colon k\hookrightarrow \CC$, and not on the auxiliary chosen continuous embedding $K\hookrightarrow \CCt$.
\end{enumerate}
\end{prop}

\begin{proof}
The first claim follows directly from Proposition \ref{prop:comparison} and the fact that the construction of \S\ref{sec:spreading} by spreading out is well-defined up to homeomorphism. Similarly, the second assertion follows from  Proposition \ref{prop:comparison} and the results of \S\ref{sec:descent}. Item (c) follows from the fact that $\Psi(X)\simeq \Psi_\log(X)$ does not depend on the auxiliary embedding for $X$ smooth and separated (as shown by the construction via Kato--Nakayama spaces).
\end{proof}

\subsection{De Rham cohomology} 
\label{sec:dr}

We explain now how $\Psi(X)$ contains information about the monodromy of $X/K$. We denote by $\widetilde{\Psi}(X)$ the fiber of the fibration $\Psi(X)$. Note that  the singular cohomology $H^*(\widetilde\Psi(X), \CC)$ carries a natural monodromy operator.

Our comparison depends on some additional data. Specifically, on top of the embedding $\iota$, we need to choose a section $k\to \cO_K$ of the projection to the residue field $\cO_K\to k$, and a basis of $\mathfrak{m}/\mathfrak{m}^2$, where $\mathfrak{m}\subseteq \cO_K$ denotes the maximal ideal. If we fix an isomorphism $K\simeq k(\!(t)\!)$ then in particular we obtain a section and a basis as above, but what we really need for the comparison is just these two pieces of data. 

Observe that the choice of basis of of $\mathfrak{m}/\mathfrak{m}^2$ induces a canonical splitting of the log point $\Spec k$ (i.e.\ an isomorphism with the standard log point $\Spec (\NN\to k)$) as follows. Every choice of a uniformizer $t\in \cO_K$ induces a chart for the log structure on $\cO_K$ and hence a splitting for the log structure on $\Spec k$. Two uniformizers $t, t'\in \mathfrak{m}$ induce the same splitting on $\Spec k$ if and only if $t' = ut$ for some $u\in 1+ \mathfrak{m}$. The quotient of the set of uniformizers of $\cO_K$ by the natural action of the group of $1$-units $1+\mathfrak{m}$ is naturally identified with $(\mathfrak{m}\setminus \mathfrak{m}^2)/\mathfrak{m}^2$, i.e.\ with possible basis elements.  

In particular, the section $k\to \cO_K$ allows us to regard everything in sight as a $k$-algebra. We denote by $D_{\cO_K}^{\rm log}$ the ring of logarithmic differential operators on $\cO_K$ relative to $k$, i.e.\ the free non-commutative ring $\cO_K\langle \partial \rangle$ with the relations
\[ 
  [\partial, f] = t \frac{\partial f}{\partial t}
\]
for $f\in \cO_K$. A $D_{\cO_K}^{\rm log}$-module which is finitely generated over $\cO_K$ is simply a finitely generated (but not neccesarily free) module with a logarithmic connection $\nabla \colon M\to M\otimes \Omega^1_{\cO_K/k}(\rm log \; \mathfrak{m})$.

\begin{thm}
  There is a functor $X \mapsto H(X)$ attaching to a scheme of finite type over $K$ an object of the derived category of $D_{\cO_K}^{\rm log}$-modules, whose cohomology modules $H^n(X)$ are finitely generated and free over $\cO_K$, together with functorial identifications
  \[ 
    H^n(X)\otimes_{\cO_K} K \simeq H^n_{\rm dR}(X/K),
    \quad
    H^n(X)\otimes_{\cO_K} \mathbf{C} \simeq H^n(\widetilde\Psi(X), \CC),
  \]
   where the first isomorphism is compatible with $\nabla$, and the second one identifies $\exp (-2\pi i\, {\rm res}_0(\nabla))$ with the monodromy operator. This functor satisfies descent for the $h$-topology.
\end{thm}

For a smooth projective variety over $K=\CCt$ this is closely related to results of Stewart--Vologodsky \cite[Section 2.2]{stewart-vologodsky}.

\begin{proof}
We construct the functor $H$ out of the category $\mathbf{GM}_{\cO_K}$ of good models and prove it satisfies $h$-descent. For a good model $\XX=(\mathscr X,\mathscr D)$, $\pi\colon \mathscr{X}\to \Spec \cO_K$, we set
\[ 
  H(\XX) = \pi_+ \cO_{\mathscr{X}} \in \mathrm{D}^b(D_{\cO_K}^{\rm log})
\]
which is a complex of $D_{\cO_K}^{\rm log}$-modules whose underlying complex of $\cO_K$-modules is $R\pi_* \Omega^{\bullet}_{\XX/\cO_K}$ (and is in fact a perfect complex), cf.\ \cite[\S 3.1]{Koppensteiner}. The $D_{\cO_K}^{\rm log}$-module structure on the cohomology corresponds to the logarithmic Gauss--Manin connection. Moreover, $H(\XX) \otimes_{\cO_K} K$ is the logarithmic de Rham cohomology of $(\mathscr{X}_K, \mathscr{D}_K)$, which coincides with the de Rham cohomology of $\mathscr{X}\setminus\mathscr{D}$.

The special fiber $H(\XX)\otimes_{\cO_K, \iota} \CC$ is a logarithmic $D$-module over the complex log point, i.e.\ an object of $\mathrm{D}^b_{\rm f.d.}(\CC[\partial])$, the full subcategory of $\mathrm{D}^b(\CC[\partial])$ spanned by complexes whose cohomology has finite dimension over $\CC$. The cohomology groups are the log de Rham cohomology of $(\mathscr{X}, \mathscr{D})_0$ relative to the log point, and the action of $\partial$ is the residue of the logarithmic Gauss--Manin connection. By \cite[Theorem~5.1.4]{achinger-ogus} (see also \cite[Theorem~6.2]{illusie-kato-nakayama}), we have an identification
\[ 
  H(\XX)\otimes_{\cO_K, \iota} \CC \simeq H^*((\mathscr{X}, \mathscr{D})_{0, \rm log} \times_{\Sone} \mathbf{R}(1), \CC), 
\]
identifying the monodromy on the right hand side with $\exp (-2\pi i \partial)$ on the left.

To show that $H$ satisfies $h$-descent, it is enough to forget the $D^{\rm log}_{\cO_K}$-module structure and consider $H(\XX)$ as a perfect complex over $\cO_K$. Further, by Lemma~\ref{lemma:perf-cx} below, it is enough to show that both $H(\XX)\otimes K$ and $H(\XX)\otimes \CC$ satisfy $h$-descent. The first functor is $R\Gamma_{\rm dR}(\mathscr{X}\setminus \mathscr{D})$, which satisfies $h$-descent (cf.\ e.g.\ \cite[\S 7]{HuberJoerder}). The latter is identified with $C^*(\widetilde \Psi(X), \CC)$, and satisfies $h$-descent because $\widetilde \Psi(X)$ does and $C^*(-, \CC)$ maps homotopy colimits of simplicial objects to homotopy limits.
\end{proof}

\begin{lemma} \label{lemma:perf-cx}
Let $E,F\in \Perf(\cO_K)$ and $\varphi\colon E\to F$ be a morphism such that both $\varphi\otimes_{\cO_K} K$ and $\varphi\otimes^{\mathbf{L}}_{\cO_K}k$ are quasi-isomorphisms. Then $\varphi$ is a quasi-isomorphism. \qed
\end{lemma}

The proof is standard and we omit it.

\section{Comparison with the \'etale homotopy type}
\label{sec:etale}

In their seminal book \cite{ArtinMazur}, Artin and Mazur construct a homotopy type $\Pi^\et(X)$ (more precisely, a pro-object in the homotopy category of simplicial sets) associated to a locally noetherian scheme $X$, called its \emph{\'etale homotopy type}. They moreover show the following comparison theorem: if $X$ is a conncted pointed scheme of finite type over $\CC$, then $\Pi^\et(X)$ receives a natural map from (the singular complex of) $X(\CC)$ which induces an equivalence on pro-finite completions. 

The goal of this section is to compare the Betti homotopy type $\Psi(X)$ of a scheme over $K$ (where $K$ is as usual equipped with an embedding $\iota\colon k\hookrightarrow \CC$ of its residue field into the complex numbers) to the \'etale homotopy type. We assume that the residue field $k$ is algebraically closed, so that ${\rm Gal}(\overline K/K) \simeq \ZZ(1)$. We use the symbol $\widehat\Pi$ to denote pro-finite completions of homotopy types. 

\begin{thm} \label{thm:etale-comparison}
  There exists a natural transformation
  \begin{equation} \label{eqn:et-comp}
    \Psi(X)\to \widehat\Pi^\et(X)
  \end{equation}
  for every scheme $X$ locally of finite type over $K$, inducing an equivalence on pro-finite completions.
\end{thm}

\begin{remark}
  For $X = \Spec K$, the above map is the standard map
  \[
    \Sone \to \widehat\Pi^\et(\Spec K) \simeq B\widehat\ZZ(1). 
  \]
  For $X$ arbitrary, denote by $\widetilde\Psi(X)$ the homotopy fiber of $\Psi(X)\to \Sone$. Then \eqref{eqn:et-comp} induces a map 
  \[ 
    \widetilde\Psi(X) \to \widehat\Pi^\et(X_{\overline K})
  \]
  which again induces isomorphisms on pro-finite completions (Proposition~\ref{prop:et-comp-fiber}). 
\end{remark}

Since one of the constructions of $\Psi(X)$ is via Kato--Nakayama spaces, it is unsurprising that the proof of the above result should rely on a logarithmic version of the Artin--Mazur comparison theorem. That is, that there is a log \'etale homotopy type associated to an fs log scheme whose underlying scheme is locally noetherian, which in the case of fs log schemes locally of finite type over $\CC$ agrees with the homotopy type of the Kato--Nakayama space up to pro-finite completion. Such a result was recently proved by Carchedi, Scherotzke, Sibilla, and the second author \cite{knvsroot}, where the role of the log \'etale homotopy type is played by the \'etale homotopy type of the associated \emph{infinite root stack}, a certain inverse system of algebraic stacks \cite{TV}. 

We give in \S\ref{ss:compare-xlog-xket} a variant of their result (with a rather easy proof), where we use the \emph{Kummer \'etale site} $(X,\Mm_X)_\ket$ of a log scheme in place of the infinite root stack. The precise relationship between the following theorem and the main result of \cite{knvsroot} can be found in \cite{loghomotopy}.

\begin{thm} \label{thm:kummer-csst}
  Let $(X,\Mm_X)$ be an fs log scheme locally of finite type over $\CC$, and let 
  \[ 
    \varepsilon\colon (X,\Mm_X)_\log \to (X,\Mm_X)_\ket
  \]
  be the natural map of sites. Then $\varepsilon$ induces a homotopy equivalence on profinite homotopy types:
  \[ 
   \widehat{\Pi}(\varepsilon)\colon \widehat \Pi((X,\Mm_X)_\log) \isomlong \widehat \Pi((X,\Mm_X)_\ket).
  \]
\end{thm}

The proof of Theorem~\ref{thm:etale-comparison} proceeds in two steps. First, as in the construction of $\Psi_\log$, we reduce to the case of $X=U$ being the trivial locus of a good model $\XX/\cO_K$. Second, if $\XX=(\mathscr X,\mathscr D)$ is a good model and $U=\mathscr X\setminus \mathscr D$, we have the following maps of sites
\[ 
  \Psi(U) = (\mathscr X,\Mm_{\mathscr D})_{0,\rm log} \to (\mathscr X,\Mm_{\mathscr D})_{0, \ket}\to  (\mathscr X,\Mm_{\mathscr D})_{k,\ket} \to (\mathscr X,\Mm_{\mathscr D})_\ket  \leftarrow % (U,\Mm_D|_U)_\ket =
  U_\et,
\]
and we show that each map induces an equivalence on pro-finite homotopy types. The first one is handled using Theorem~\ref{thm:kummer-csst}, the second map comes from an extension of algebraically closed fields of characteristic $0$, the third one is a variant of proper base change, and the last one is a form of ``purity''.

\subsection{Review of \'etale homotopy}

We refer to \cite{ArtinMazur,Friedlander, carchedi} for the details on \'etale homotopy, giving here only a brief outline. Consider a site $C$ satisfying the following properties:
\begin{itemize}
  \item $C$ admits  finite products, fiber products, and finite coproducts,
  \item $C$ is \emph{subcanonical}, i.e.\ representable presheaves in $C$ are sheaves,
  \item $C$ is \emph{locally connected}, i.e. every object is a coproduct of objects which do not admit  non-trivial coproduct decompositions.
\end{itemize} 
These properties are satisfied, for example, by the \'etale site of a locally noetherian scheme. For such $C$, every object has a well-defined set of connected components, which gives rise to a connected component functor 
\[ 
  \pi\colon C\to {\rm Set}.
\]
If $U_\bullet\colon \Delta^{\rm op}\to C$ is a hypercovering (of the final object), the simplicial set $\pi(U_\bullet)$ is a ``nerve'' of $U_\bullet$, to be thought of as an approximation to the homotopy type of $C$. Hypercoverings in $C$ can be organized into a cofiltering category ${\rm HR}(C)$ where maps are homotopy classes of morphisms, and therefore 
\[
  \Pi(C) := \{\pi(U_\bullet)\}_{U_\bullet \in {\rm HR}(C)} \in \text{pro-}H
\]
is a well-defined object in the category of pro-homotopy types, functorial in $C$.

\begin{remark}
With a bit more work, in the case of the \'etale site of a locally noetherian scheme, one can upgrade this construction to yield a pro-object in the category of simplicial sets, see \cite{Friedlander}. Moreover, this machinery was recently further refined and generalized to arbitrary higher stacks by Carchedi \cite{carchedi}. We will mostly stick to the ``classical language'' of Artin--Mazur.
\end{remark}

For technical reasons, it is often important to assume that $C$ is connected (the final object does not admit nontrivial coproduct decompositions) and pointed (i.e.\ with a chosen point $p\colon {\rm Set}\to C$). In this case, considering pointed hypercoverings above yields an object $\Pi(C)$ in the category $\text{pro-}H_0$, where $H_0$ is the homotopy category of pointed simplicial sets.

\begin{defin}
  Let $X$ be a locally noetherian pointed scheme. The \emph{\'etale homotopy type} of $X$, denoted $\Pi^\et(X)$, is the pro-object $\Pi(X_\et)$ of constructed above associated to the \'etale site of $X$ with the chosen base point.
\end{defin}

An object in $\text{pro-}H_0$ is called \emph{pro-finite} if all of its homotopy groups are finite. The inclusion of pro-finite objects into $\text{pro-}H_0$ admits a left adjoint $(-)^\wedge$, the \emph{pro-finite completion} functor. If $X$ is a connected pointed geometrically unibranch noetherian  scheme, then $\Pi^\et(X)$ is pro-finite \cite[Theorem 11.1]{ArtinMazur}. An important criterion detecting whether a map of sites induces an equivalence of pro-finite homotopy types is the following.

\begin{thm} \label{thm:artinmazur-coh-comparison}
  Let $f\colon C\to D$ be a morphism of pointed connected sites satisfying the above conditions. Suppose that
  \begin{enumerate}[i.]
    \item both $C$ and $D$ have finite local cohomological dimension with respect to the class of finite groups (cf.\ \cite[Definition~8.17]{ArtinMazur}),
    \item for every finite group $G$, $f$ induces an isomorphism on non-abelian cohomology
    \[  
      f^*\colon H^1(D, G) \isomlong H^1(C, G),
    \]
    \item for every local system $F$ of finite abelian groups on $D$, $f$ induces an isomorphism on cohomology
    \[  
      f^*\colon H^*(D, F) \isomlong H^*(C, f^* F).
    \]
  \end{enumerate}
  Then the induced map of pro-finite completions
  \[
    f\colon \widehat\Pi(C) \to \widehat\Pi(D)
  \]
  is an isomorphism in $\text{pro-}H_0$.
\end{thm}

\begin{proof}
Combine \cite[Theorem~4.3]{ArtinMazur} with \cite[Theorem~12.5]{ArtinMazur}.
\end{proof}

For brevity, we will call a map of (arbitrary) sites $f\colon C\to D$ a \emph{$\natural$-isomorphism} if it satisfies conditions $ii.$ and $iii.$ above (this is slightly different from the terminology used in \cite{ArtinMazur}). As an example application, the comparison theorem \cite[Exp.\ XVI]{SGA4} states that the map comparing the classical topology to the \'etale topology of a scheme locally of finite type over $\CC$ is a $\natural$-isomorphism, which by Theorem~\ref{thm:artinmazur-coh-comparison} easily implies the Artin--Mazur comparison theorem:

\begin{thm}[{\cite[Theorem 12.9]{ArtinMazur}}]
  Let $X$ be a connected pointed scheme of finite type over $\CC$. Then the natural map
  \[ 
  X_\an \to \Pi^\et(X)
  \]
  induces an equivalence of pro-finite completions.
\end{thm}

To apply the criterion of Theorem~\ref{thm:artinmazur-coh-comparison} in our situation, we will need the following lemma.

\begin{lemma}
  Let $X$ be a scheme of finite type over $K$. Then $X_\et$ has finite local cohomological dimension with respect to the class of finite groups.
\end{lemma}

\begin{proof}
Let $d=\dim X$. We will show that if $F$ is a constructible \'etale sheaf on $X$ then $H^n(X, F) = 0$ for $n>2d+1$, which implies the claim. Let $G={\rm Gal}(\overline{K}/K) \simeq \widehat\ZZ(1)$ (recall that we are assuming that $k$ is algebraically closed), let $\overline X = X\otimes \overline{K}$, and let $\overline F$ be the pull-back of $F$ to $\overline X$, endowed with the natural continuous $G$-action. We have the spectral sequence
\[
  E^{p,q}_2 = H^p(G, H^q(\overline X, \overline F)) \quad \Rightarrow \quad H^{p+q}(X, F).
\]
Since $G$ has cohomological dimension $1$ and $\overline X$ has cohomological dimension $\leq 2d$, we have $E^{p,q}_2 = 0$ for $p>1$ or $q>2d$, and we conclude that $H^n(X, F) = 0$ for $n>2d+1$, as desired.
\end{proof}

We will also need the fact that the \'etale homotopy type satisfies a form of simplicial descent in the $h$-topology.

\begin{prop}\label{prop:etale.homotopy.descent}
  Let $V_\bullet\to U$ be an $h$-hypercovering of schemes locally of finite type over $K$. Then the induced map
  \[
    \hocolim \widehat\Pi^\et(V_\bullet) \to \widehat\Pi^\et(U)
  \]
  is an equivalence.
\end{prop}

\begin{proof}
We apply \cite[Theorem 1.19]{berner} with the \'etale and the $h$-topology of the category ${\bf Sch}^{\rm lft}_K$. What we have to check in order to conclude the desired hyperdescent property is that for every locally finite type scheme $X$ over $K$, the profinite shape of the \'etale site $X_\et$ is the same as the profinite shape of the $h$-site $X_h$.

Denote by $\delta\colon X_h\to X_\et$ the natural map of sites. The claim then follows from the fact that for an \'etale local system $F$ of finite abelian groups we have $H^i(X_\et,F)\simeq H^i(X_h,\delta^{-1}F)$ for all $i\in \NN$ \cite[Tag 0EWH]{stacks-project}, and that for every finite group $G$ we have $H^1(X_\et,G)\simeq H^1(X_h,G)$. The latter statement about torsors follows from the fact that universally submersive morphisms are of effective descent for finite \'etale morphisms \cite[Corollary 5.18]{rydh}.
\end{proof}

\subsection{Review of the Kummer \'etale topology}

The ``correct'' analogue of the \'etale site $X_\et$ for a fs log scheme $(X,\Mm_X)$ is the Kummer \'etale site $(X,\Mm_X)_\ket$. See \cite{IllusieFKN} for a survey of the Kummer \'etale topology.

\begin{defin} \mbox{}
  \begin{enumerate}[(i)]
    \item A morphism $f\colon (Y,\Mm_Y)\to (X,\Mm_X)$ of fs log schemes is \emph{Kummer \'etale} if it is log \'etale and the cokernel of $f^* \overline \Mm^{\rm gp}_{X}\to \overline\Mm^{\rm gp}_Y$ is torsion.
    \item A family $\{f_i\colon (Y_i,\Mm_{Y_i})\to (X,\Mm_X)\}$ is a Kummer \'etale cover if the maps $f_i$ are Kummer \'etale and jointly surjective.
    \item The \emph{Kummer \'etale site} $(X,\Mm_X)_\ket$ is the category of Kummer \'etale fs log schemes over $(X,\Mm_X)$, endowed with the topology induced by the Kummer \'etale covers.
  \end{enumerate}
\end{defin}

Since strict \'etale maps are Kummer \'etale, there is a natural morphism of  sites 
\[
  \sigma\colon (X,\Mm_X)_\ket\to X_\et.
\]

\begin{lemma}\label{lem:log.properbasechange}
  Let $f\colon (X,\Mm_X)\to (Y,\Mm_Y)$ be a strict map of fs log schemes, and let $F$ be sheaf of torsion abelian groups on $(Y,\Mm_Y)_\ket$. Then the base change morphism
  \[ 
    f^* R\sigma_{Y*} F \to R\sigma_{X*} f^* F
  \]
  is an isomorphism. 
\end{lemma}

\begin{proof}
This is a special case of Nakayama's log proper base change \cite[Theorem~5.1]{NakayamaMathAnn}.
\end{proof}

\begin{thm}[{Log Purity \cite{FujiwaraGabber}, \cite[2.0.1,2.0.5]{Nakayama}}] \label{thm-ket-log-regular}
Let $(X,\Mm_X)$ be a log regular fs log scheme over $\mathbf{Q}$, and let $U = (X,\Mm_X)_{\rm triv}$ be its trivial locus. Then the natural map
\[  
  U_\et = (U,\Mm_X|_U)_\ket \to (X,\Mm_X)_\ket  
\]
is a $\natural$-isomorphism. \qed
\end{thm}

\begin{lemma} \label{lemma:xket-properties}
  Let $(X,\Mm_X)$ be an fs log scheme whose underlying scheme is locally noetherian. 
  \begin{enumerate}[(a)]
    \item The site $(X,\Mm_X)_\ket$ satisfies the conditions necessary for $\Pi((X,\Mm_X)_\ket)$ to be defined, and if the underlying scheme $X$ is connected, then the site $(X,\Mm_X)_\ket$ is connected.
    \item Let $r$ be the supremum of the ranks of the stalks of $\overline\Mm_X^{\rm gp}$, and assume that this is finite. Then
      \[
        \cd (X,\Mm_X)_\ket \leq \cd X_\et + r.
      \]
      where $\cd$ denotes cohomological dimension with respect to finite groups.
  \end{enumerate}
\end{lemma}

\begin{proof}
The first item is immediate. Item $(b)$ reduces to the case of a logarithmic point over an algebraically closed field, thanks to Lemma \ref{lem:log.properbasechange}. The statement for a log point is clear from the fact that in this case the topos associated to $X_\ket$ is equivalent to the classifying topos of the group $\widehat{\ZZ}^r$.
\end{proof}

\begin{thm}[{a variant of Log Proper Base Change}] \label{thm-logpbc}
  Let $X$ be a proper scheme over $\Spec A$ where $A$ is a henselian local ring with residue field $k$. Let $X_0 = X\otimes_A k$ and let $\Mm_X\to \cO_X$ be an fs log structure on $X$. Then the natural morphism
  \[ 
    (X_0,\Mm_X|_{X_0})_{\ket} \to (X,\Mm_X)_\ket
  \]
  is a $\natural$-isomorphism.
\end{thm}

\begin{proof}
This  follows from Theorem \ref{thm:artinmazur-coh-comparison}, thanks to Lemma \ref{lemma:xket-properties},  \cite[Proposition 6.3]{IllusieFKN} and \cite[Proposition 5.1]{OlssonKummer}.
\end{proof}

\subsection{\texorpdfstring{Comparing $(X,\Mm_X)_\log$ and $(X,\Mm_X)_\ket$}{Comparing Xlog and Xket}}
\label{ss:compare-xlog-xket}

Let us now consider the case of an fs log scheme $(X,\Mm_X)$ with $X$ locally of finite type over $\CC$, which is the situation studied extensively by Kato and Nakayama in \cite{KN}. In this case, if $(U,\Mm_U)\to (X,\Mm_X)$ is Kummer \'etale then the associated $(U,\Mm_U)_\log\to (X,\Mm_X)_\log$ is a local homeomorphism, and consequently there is a natural map of sites
\[ 
  \varepsilon\colon (X,\Mm_X)_\log \to (X,\Mm_X)_\ket.
\]

\begin{thm}[{\cite[Theorem 0.2(I)]{KN}}] \label{thm:kn}
  Let $(X,\Mm_X)$ be an fs log scheme locally of finite type over $\CC$. Then for every constructible (see \cite[Definition~2.5.1]{KN}) sheaf of torsion abelian groups $F$ on $(X,\Mm_X)_\ket$, the map $\varepsilon$ induces isomorphisms
  \[ 
    \varepsilon^* \colon H^q((X,\Mm_X)_\ket, F) \isomlong H^q((X,\Mm_X)_\log, \varepsilon^{-1}F)
  \]
  for all $q\geq 0$. 
\end{thm}

To show that $\varepsilon$ is a $\natural$-isomorphism, we need to compare non-abelian $H^1$ as well.

\begin{prop} \label{prop}
  Let $(X,\Mm_X)$ be an fs log scheme locally of finite type over $\CC$. Then for every finite group $G$, the pull-back map
  \[
    \varepsilon^* \colon H^1((X,\Mm_X)_\ket, G)\to H^1((X,\Mm_X)_\log, G)
  \]
  is an isomorphism.
\end{prop}

\begin{proof}
We need an easy abstract lemma, which follows e.g.\ by applying \cite[Proposition~III~3.1.3, p.~323]{giraud} to $A = f^* G$.

\begin{lemma}
  Let $f\colon X\to Y$ be a map of sites, and let $G$ be a sheaf of groups on $Y$. For a~sheaf of groups $F$ on $X$, denote by $R^1 f_*  F$ the sheaf of pointed sets associated to the presheaf $U \mapsto H^1(f^{-1}(U), F)$ (isomorphism classes of $F$-torsors) on $Y$. Suppose that 
  \begin{enumerate}[(a)]
    \item $R^1 f_* (f^* G) = \star_Y$, the trivial sheaf of pointed sets on $Y$,
    \item $G\to f_* f^* G$ is an isomorphism.
  \end{enumerate}
  Then the pull-back map (induced by pull-back of torsors)
  \[ 
    f^* \colon H^1(Y, G) \to H^1(X, f^* G)
  \]
  is bijective.\qed 
\end{lemma}

To apply the lemma to the situation of the proposition (and the constant sheaf $G$), we first check $(b)$. This is in fact a statement about sheaves of sets, and amounts to checking that $(U,\Mm_U)_\log$ is connected whenever $U$ is connected as an object of $(X,\Mm_X)_\ket$, and it follows directly from Theorem \ref{thm:kn} with $F=\ZZ/n\ZZ$ and $q=0$.

Now we check $(a)$; we need to show that the presheaf of pointed sets on the Kummer \'etale site associating to $(U,\Mm_U)\to (X,\Mm_X)$ the set $H^1((U,\Mm_U)_\log, G)$ of isomorphism classes of $G$-torsors on $(U,\Mm_U)_\log$ has trivial sheafification. In other words: given a $G$-torsor $T$ on $(U,\Mm_U)_\log$, there exists a Kummer \'etale and surjective $(V,\Mm_V)\to (U,\Mm_U)$ such that the pull-back of $T$ to $V$ is constant. (This can be seen as a variant of \cite[Lemma~2.5.2]{KN}.) To this end, we can assume that there is a chart $U\to \Spec \ZZ[P]$. Let $N$ be the exponent of $G$, and consider the following pull-back
\[ 
  \xymatrix{
    W\ar[r]\ar[d] & \Spec \ZZ[P]\ar[d]^{N\colon P\to P} \\
    U \ar[r] & \Spec \ZZ[P]. 
  }
\]
Then $(W,\Mm_W)\to (U,\Mm_U)$ is a Kummer \'etale covering ($W$ is equipped with the log structure induced by the upper horizontal morphism), and the pull-back of $T$ to $(W,\Mm_W)_\log$ is constant on the fibers of $\tau\colon (W,\Mm_W)_\log\to W_\an$, therefore (by properness of $\tau$!) $T \simeq \tau^* S$ for some $G$-torsor $S$ on $W_\an$ (in fact $S = \tau_* T$). By the comparison theorem \cite[Exp.\ XVI]{SGA4}, this is identified with an \'etale $G$-torsor on $W_\et$. Passing to a (strict) \'etale covering $(V,\Mm_V)\to (W,\Mm_W)$, we can thus make the pull-back of $S$ to $V_\an$ trivial, and then $T$ becomes trivial on $(V,\Mm_V)_\log$, as desired.
\end{proof}

We are now ready to prove Theorem~\ref{thm:kummer-csst}.

\begin{proof}[Proof of Theorem~\ref{thm:kummer-csst}]
Combining Theorem~\ref{thm:kn} with Proposition~\ref{prop}, we see that the map $\varepsilon\colon (X,\Mm_X)_\log\to (X,\Mm_X)_\ket$ is a $\natural$-isomorphism. Moreover, by Lemma \ref{lemma:xket-properties}(b) we see that $(X,\Mm_X)_\ket$ has finite local cohomological dimension with respect to the class of finite groups. We conclude by Theorem~\ref{thm:artinmazur-coh-comparison}.
\end{proof}

\subsection{\texorpdfstring{Comparing $\Psi(X)$ and $\Pi^\et(X)$}{Comparing Psi(X) and Pi et (X)}}

We can now put every piece of the argument together. We first prove our comparison result on good models.

\begin{prop} \label{prop:etale-comparison-good-model}
  Let $\XX=(\mathscr X,\mathscr D)$ be a good model over $\cO_K$ and let $U = \mathscr X
 \setminus \mathscr D$. Then the maps
  \[ 
    \Psi(U) =(\mathscr X,\Mm_{\mathscr D})_{0,\rm log} \to (\mathscr X,\Mm_{\mathscr D})_{0,\ket}\to (\mathscr X,\Mm_{\mathscr D})_{k,\ket} \to (\mathscr X,\Mm_{\mathscr D})_\ket \leftarrow % (U,\Mm_D|_U)_\ket =
    U_\et,
  \]  
  induce isomorphism on the associated pro-finite homotopy types. Therefore we obtain a morphism
  \[ 
    \Psi(U) \to \widehat{\Pi}^\et(U)
  \]
  which induces an isomorphism upon pro-finite completion.
\end{prop}

\begin{proof}
Combine Theorem~\ref{thm:kummer-csst}, Theorem~\ref{thm-logpbc}, and Theorem~\ref{thm-ket-log-regular} with Theorem~\ref{thm:artinmazur-coh-comparison}. 
\end{proof}

Finally, we use a descent argument to extend the comparison to arbitrary locally finite type schemes over $K$, finishing the proof of the comparison result.

\begin{proof}[Proof of Theorem~\ref{thm:etale-comparison}]
Using Proposition~\ref{prop:extend-cosheaf}, since both $\Psi$ and $\widehat{\Pi}^\et$ are $h$-hypercosheaves (Propositions \ref{prop:descent} and \ref{prop:etale.homotopy.descent}), the natural transformation of Proposition~\ref{prop:etale-comparison-good-model} extends uniquely to all schemes locally of finite type over $K$, and gives an equivalence up to profinite completion.
\end{proof}

We finish by deducing a similar comparison between the fiber $\widetilde \Psi(X)$ of $\Psi(X)\to \Sone$ and the \'etale homotopy type of $X_{\overline K}$.

\begin{prop} \label{prop:et-comp-fiber}
There exists a natural transformation $\widetilde \Psi(X) \to \widehat\Pi^\et(X_{\overline K})$, fitting inside a~commutative square
\[
  \xymatrix{
    \widetilde\Psi(X) \ar[r] \ar[d] & \widehat\Pi^\et(X_{\overline K}) \ar[d] \\
    \Psi(X) \ar[r] &  \widehat\Pi^\et(X)
  }
\]
and inducing an equivalence on pro-finite completions.
\end{prop}

\begin{proof}
Since both source and target of the required natural transformations satisfy Zariski descent, we can restrict our attention to separated schemes of finite type over $K$; this will be needed for some finiteness properties at the end of the argument below.

Fix a uniformizer $t$ of $K$ and let $K_n = K(\sqrt[n]{t})$, so that $\overline K = \bigcup_{n\geq 1} K_n$. For a scheme $X$ over $K$, we denote by $X_n$ its base change to $K_n$. The construction in Theorem~\ref{thm:etale-comparison} is compatible with base change in the sense that for $n$ dividing $m$ we have a commutative square
\[ 
  \xymatrix{
    \Psi(X_m) \ar[r] \ar[d] & \widehat{\Pi}^\et(X_m) \ar[d] \\
    \Psi(X_n) \ar[r] & \widehat{\Pi}^\et(X_n).
  }
\]
The space $\widetilde\Psi(X)$ maps compatibly to all $\Psi(X_n)$, and the same holds for $\widehat\Pi^\et(X_{\overline K})$ and the $\widehat\Pi^\et(X_n)$. Passing to the homotopy limit over $n$, we obtain a diagram
\[ 
  \xymatrix{
    \widetilde\Psi(X) \ar[d]_{(1)} & \widehat\Pi^\et(X_{\overline K}) \ar[d]^{(3)} \\
    \holim \Psi(X_n) \ar[r]_{(2)} & \holim \widehat\Pi^\et(X_n).
  }
\]

We claim that the maps (1), (2) and (3) above induce isomorphisms on the cohomology groups of local systems of finite groups, including non-abelian $H^1$. This will imply that the induced maps on pro-finite completions are equivalences, and the required assertion will follow. Note that for the two homotopy limits above, the cohomology groups in question are the direct limits of the cohomology groups at finite levels. For (1), see \cite[Lemma~4.3.3]{PiotrThesis}. For (2), this follows from Theorem~\ref{thm:etale-comparison}. The assertion for (3) is a standard fact in \'etale cohomology \cite[Exp.\ VII, Th\'eor\`eme~5.7]{SGA4_2}.
\end{proof}

\section{Rigid geometry}
\label{s:rigid}

Let $K = {\rm Frac}\, \cO_K$ be a complete discretely valued field whose residue field $k$ is endowed with an embedding $\iota\colon k\hookrightarrow \CC$. In this section, we shall construct a topological realization functor for rigid analytic spaces over $K$. 

As the spreading out techniques of \S\ref{sec:spreading} are not available in this context, our course of action is to use the log geometry construction of \S\ref{sec:log.geometry} by finding suitable good formal models. This works only in the smooth case, and a different idea is needed to prove independence of the model. Such a method is supplied by the Weak Factorization Theorem in the form obtained recently by Abramovich and Temkin \cite{weakfactorization} combined with a key computation of the fibers of the map of Kato--Nakayama spaces induced by a simple blowup (see \S\ref{sec:key.calculation}). 

As we were unable to show that the functor thus obtained on smooth rigid analytic spaces (or, more generally, simple normal crossings pairs) satisfies $h$-descent, we do not know that it extends to the singular case.

\subsection{Generalities on localizations of categories} 
\label{ss:localiz}

Before we start discussing rigid analytic spaces, we need to gather some standard facts about ways of formally inverting morphisms in a~category.

Let $(\Cc, W)$ be a category with equivalences, i.e.\ a category $\Cc$ with a subcategory $W$ containing all isomorphisms and satisfying the two-out-of-three property: if $f$ and $g$ are composable arrows in $\Cc$ and two of $f$, $g$, $gf$ are in $W$ then so is the third. There exists a functor 
\[
  \Cc\to \Cc[W^{-1}]
\]
which is initial among all functors $\Cc\to \mathscr{D}$ sending morphisms in $W$ to isomorphisms; the category $\Cc[W^{-1}]$ is called the \emph{localization} of $\Cc$ in $W$. 

Similarly, there exists a functor to an $\infty$-category 
\[ 
  \Cc\to L\Cc
\] 
which is initial among the functors $\Cc\to \mathscr{D}$ to $\infty$-categories $\mathscr{D}$ sending morphisms $W$ to equivalences in $\mathscr{D}$; the $\infty$-category $L\Cc$ is called the \emph{$\infty$-categorical localization of $(\Cc, W)$}, and can be realized concretely as the simplicial nerve of the simplicial localization constructed by Dwyer and Kan \cite{DwyerKan}. By the universal property, there are induced functors $L\Cc\to \Cc[W^{-1}]$ and $\Cc[W^{-1}]\to {\rm Ho}(L\Cc)$; the latter functor is an equivalence.

We say that $(\Cc, W)$ admits a \emph{calculus of right fractions} \cite{GabrielZisman} \cite[7.1]{DwyerKan} \cite[Tag 04VC]{stacks-project} if:
\begin{enumerate}[a)]
  \item every pair of solid arrows as below with $u\in W$ can be completed to a commutative square
  \[
    \xymatrix @C=1em @R=1em{
      & \bullet \ar@{.>}[dl]_{v\in W} \ar@{.>}[dr] &  \\
      \bullet \ar[dr] & & \bullet \ar[dl]^{u\in W} \\
      & \bullet &
    }
  \]
  with $v\in W$, and
  \item for every pair of parallel morphisms $f,g\colon X\to Y$ in $\Cc$ and every map $u\colon Y\to Z$ in $W$ such that $uf=ug$ there exists a map $v\colon W\to X$ in $W$ such that $fv=gv$.
\end{enumerate}
If $W$ admits a calculus of right fractions, by \cite[Tag 04VH]{stacks-project} the localization $\Cc[W^{-1}]$ admits a useful explicit description; the objects of $\Cc[W^{-1}]$ are the objects of $\Cc$, and the  morphisms $c\to c'$ in $\Cc[W^{-1}]$ are equivalence classes of ``roofs'' $c\leftarrow c_0 \to c'$ with the backwards map in $W$, where $c\leftarrow c_0\to c'$ and $c\leftarrow c_1 \to c'$ are equivalent if there exists a third $c\leftarrow c_2\to c'$ and maps $c_0\leftarrow c_2\to c_1$ making the resulting diagram commute. Given $c\leftarrow c_0 \to c'$ and $c'\leftarrow c_1\to c''$, applying axiom a) to $c_0\to c' \leftarrow c_1$ gives $c_0\leftarrow c_2\rightarrow c_1$, and the composition is $c \leftarrow c_0\leftarrow c_2 \rightarrow c_1 \rightarrow c''$.  

\begin{thm}[{Dwyer--Kan, \cite[\S 7]{DwyerKan}}] \label{thm:dwyer-kan}
  Suppose that $(\Cc, W)$ admits a calculus of right fractions. Then
  \[ 
    L\Cc \to \Cc[W^{-1}]
  \]
  is an equivalence. In other words, $L\Cc$ is a $1$-category. \qed
\end{thm}

The following definition expresses the abstract properties of the statement of the Weak Factorization Theorem in birational geometry. We will apply it later on to $\Cc$ consisting of good formal models, $W$ admissible blowups, $W'$ simple admissible blowups.

\begin{defin} \label{def:weak.fact}
  Let $(\Cc, W)$ be a category with equivalences, and let $W'\subseteq W$ be a class of morphisms. We say that $W'$ has the \emph{weak factorization property} if for every map $\pi\colon X'\to X$ in $W$ there exists a zigzag
  \begin{equation} \label{eqn:zigzag}
   \xymatrix{
      & X_1  \ar[dr]^{\beta_1} \ar[dl]_{\alpha_1} &  & \ar[dl] & \ldots \ar@{}[d]|-{\displaystyle \ldots}  &  \ar[dr]  & &  X_m \ar[dr]^{\beta_m} \ar[dl]_{\alpha_m} \\
      X'=X_0 & & X_2 &  & \ldots &  & X_{m-1} & & X
    }
  \end{equation}
  with $\alpha_1, \beta_1, \ldots, \alpha_m, \beta_m$ in $W'$, and maps $\pi_i\colon X_i\to X$ ($i=0, \ldots, m$) with $\pi_0 = \pi$, such that ``the resulting diagram commutes in $\Cc[W^{-1}]$ after reversing the maps $\alpha_i$'', i.e.\ for every $i=0, \ldots, m$, the following equality holds in $\Hom_{\Cc[W^{-1}]}(X_i, X)$:
  \begin{equation} \label{eqn:pi-zigzag}
    \pi_i = 
    \begin{cases}
      \beta_m \circ \alpha_m^{-1} \circ \cdots \circ \beta_{i+1} \circ \alpha_{i+1}^{-1} & \text{for }i\text{ even} \\
      \beta_m \circ \alpha_m^{-1} \circ \cdots \circ \beta_{i+1} \circ \alpha_{i+1}^{-1}\circ \beta_{i} & \text{for }i\text{ odd}.
    \end{cases}
  \end{equation}
\end{defin}
If $\Cc\to \Cc[W^{-1}]$ is faithful, \eqref{eqn:pi-zigzag} is equivalent to the equalities in $\Hom_\Cc(X_i, X)$
\begin{equation} \label{eqn:pi-zigzag2}
  \pi_{i-1}\circ\alpha_{i-1} = \pi_{i} = \pi_{i+1}\circ\beta_{i} \quad \quad \text{ for }i\text{ odd}.
\end{equation}

\begin{lemma} \label{lemma:weak.fact}
  Let $(\Cc, W)$ be a category with equivalences and let $W'\subseteq W$ be a class of morphisms with the weak factorization property. Suppose that $\Cc\to \Cc[W^{-1}]$ is faithful. Then the induced map 
  \[
    \Cc[(W')^{-1}]\to \Cc[W^{-1}]
  \] 
  is an equivalence, i.e.\ every functor $\Cc\to \mathscr{D}$ sending all morphisms in $W'$ to isomorphisms in $\mathscr{D}$ sends all morphisms in $W$ to isomorphisms in $\mathscr{D}$.
\end{lemma}

\begin{proof}
Let $F\colon \Cc\to \mathscr{D}$ be a functor sending all maps in $W'$ to isomorphisms, and let $\pi:X'\to X$ be a morphism in $W$. We have to show that $F(\pi)$ is an isomorphism. Repeatedly applying \eqref{eqn:pi-zigzag2} and using the assumption that $F(\alpha_i)$ are isomorphisms, we obtain
% Two-line version (for bigger font size)
% \begin{align*} 
%   F(\pi) &= F(\pi_1)\circ F(\alpha_1)^{-1} = F(\pi_2)\circ F(\beta_1) \circ F(\alpha_1)^{-1} = \cdots \\
%   &= F(\beta_m)\circ F(\alpha_m)^{-1}\circ\ldots\circ F(\beta_1) \circ F(\alpha_1)^{-1}, 
% \end{align*}
\[ 
  F(\pi) = F(\pi_1)\circ F(\alpha_1)^{-1} = F(\pi_2)\circ F(\beta_1) \circ F(\alpha_1)^{-1} = \cdots 
  = F(\beta_m)\circ F(\alpha_m)^{-1}\circ\ldots\circ F(\beta_1) \circ F(\alpha_1)^{-1}, 
\]
which is invertible since the $F(\beta_i)$ are.
\end{proof}

\begin{defin} \label{def:localization}
  We call a functor $F:\Cc\to \mathscr{D}$ a \emph{localization} if $F$ is faithful and if we denote by $W_F\subseteq \Cc$ the class of all morphisms in $\Cc$ which $F$ sends to isomorphisms in $\mathscr{D}$, then the induced functor $\Cc[W_F^{-1}] \to \mathscr{D}$ is an equivalence.
\end{defin}

We omit the standard but somewhat lengthy proof of the following criterion for being a localization.

\begin{lemma}[{cf.\ \cite[Proof of Theorem~4.1]{FormalRigidI}}] \label{lemma:detect-localiz}
Let $F:\Cc\to \mathscr{D}$ be a functor, and let $W_F$ be the class of all morphisms in $\Cc$ which $F$ sends to isomorphisms in $\mathscr{D}$. Suppose that 
\begin{enumerate}[i.]
  \item $F$ is faithful and essentially surjective, 
  \item $F$ is ``weakly full with fixed target'' i.e. for $c'\in \Cc$ and $f\colon d\to F(c')\in \mathscr{D}$ there exists a~$g:c\to c'\in \Cc$ and an isomorphism $u:F(c)\simeq d$ such that $F(g)=f$ under this isomorphism,
  \item the fiber categories $F^{-1}(d)$ (which are posets because $F$ is faithful) are cofiltering (any two objects are dominated by a third one).
\end{enumerate}
Then $W_F$ admits a calculus of right fractions and $F$ is a localization, i.e.\ the induced functor $\Cc[W_F^{-1}]\to \mathscr{D}$ is an equivalence.  \qed
\end{lemma}

\subsection{Review of rigid analytic spaces and formal schemes}
\label{ss:raynaud}

Let $K$ be a complete discretely valued field with valuation ring $\cO_K$ and residue field $k$. By a \emph{rigid analytic space} over $K$ we shall mean a rigid analytic variety in the sense of \cite[Definition~9.3.1/4]{BGR} or \cite[Definition~5.3/1]{Bosch}, i.e.\ a locally ringed $G$-space admitting an admissible covering by affinoid rigid analytic spaces of the form ${\rm Sp}\, K\langle x_1, \ldots, x_n\rangle/(f_1, \ldots, f_r)$. We denote the category of rigid analytic spaces over $K$ by $\mathbf{Rig}_K$.   

If $X$ is a scheme locally of finite type over $K$, one can construct a rigid analytic space $X_{\rm an}$, called its \emph{analytification}. If $X=\Spec A$ is affine, then $X_{\rm an}$ represents the functor $\mathcal{Y}\mapsto \Hom_K(A, \Gamma(\mathcal{Y}, \cO_{\mathcal{Y}}))$; the general case follows by gluing. This defines the analytification functor
\[ 
  (-)_{\rm an} \colon \mathbf{Sch}^{\rm lft}_K \to \mathbf{Rig}_K.
\]

Let $\mathfrak{X}$ be a formal scheme locally of topologically finite type over $\Spf \cO_K$, that is, locally of the form $\Spf \cO_K\langle x_1, \ldots, x_n\rangle /(f_1, \ldots, f_r)$. We denote the category of such formal schemes by $\mathbf{FSch}_{\cO_K}$. If $\mathfrak{X}= \Spf A$ for a topologically of finite type $\cO_K$-algebra $A$, then $A\otimes K$ is an affinoid $K$-algebra; the associated rigid analytic space is the affinoid space $\mathfrak{X}_{\rm rig} = {\rm Sp}\, (A\otimes K)$. Again, for a general $\mathfrak{X}$, one can construct $\mathfrak{X}_{\rm rig}$ by gluing, obtaining a functor
\[ 
  (-)_{\rm rig} \colon \mathbf{FSch}_{\cO_K} \to \mathbf{Rig}_K.
\]

A formal scheme locally of topologically finite type is called \emph{admissible} if it is flat over $\cO_K$; the functor $(-)_{\rm rig}$ restricted to the full subcategory $\mathbf{FSch}^{\rm adm}_{\cO_K}\subseteq \mathbf{FSch}_{\cO_K}$ of admissible formal schemes is faithful. A morphism $\mathfrak{X}'\to \mathfrak{X}$ between admissible formal schemes is called an \emph{admissible blowup} if it is isomorphic to the formal blowing up of $\mathfrak{X}$ in an open coherent ideal; this is equivalent to $\mathfrak{X}'_{\rm rig}\to \mathfrak{X}_{\rm rig}$ being an isomorphism. We denote by $W\subseteq \mathbf{FSch}^{\rm adm}_{\cO_K}$ the category of admissible blowups.

Let $\mathbf{FSch}^{\rm adm, qc}_{\cO_K}$ (resp.\ $\mathbf{Rig}^{\rm qcqs}_K$) denote the full subcategory of $\mathbf{FSch}^{\rm adm}_{\cO_K}$ (resp.\ $\mathbf{Rig}_K$) consisting of quasi-compact (resp.\ and quasi-separated) objects.

\begin{thm}[{Raynaud \cite{Raynaud,FormalRigidI}}] \label{thm:raynaud}
  The functor $(-)_{\rm rig}\colon \mathbf{FSch}^{\rm adm, qc}_{\cO_K} \to \mathbf{Rig}^{\rm qcqs}_K$  satisfies the conditions of Lemma~\ref{lemma:detect-localiz}, and hence induces an equivalence
  \[ 
    \mathbf{FSch}^{\rm adm, qc}_{\cO_K}[W^{-1}] \isomlong \mathbf{Rig}^{\rm qcqs}_K.
  \]
\end{thm} 

\subsection{Good formal models}
\label{ss:good-models-rigid}

Recall from Definition~\ref{def:good.model} that good models of $K$-schemes over $\cO_K$ are assumed to be proper; otherwise one would allow e.g.\ models with empty special fiber. We omit this assumption in the case of rigid analytic spaces, for two reasons: first, different models are related by admissible blowups, which makes comparing them easier. Second, compactifications rarely exist. 

\begin{defin}\label{def:good.model-rig}\mbox{ }
\begin{enumerate}[(a)]
  \item A \emph{rigid snc pair} is a pair $(\mathcal{Y}, \mathcal{D})$ consisting of a smooth rigid analytic space $\mathcal{Y}$ and a divisor $\mathcal{D}\subseteq \mathcal{Y}$ with simple normal crossings. A morphism of snc pairs $\pi\colon (\mathcal{Y}', \mathcal{D}')\to (\mathcal{Y}, \mathcal{D})$ is a map $\pi\colon \mathcal{Y}'\to \mathcal{Y}$ such that $\pi^{-1}(\mathcal{D})\subseteq \mathcal{D}'$. This defines the category $\mathbf{RigPairs}_K$ of rigid snc pairs over $K$. It contains the category $\mathbf{SmRig}_K$ as a full subcategory (by taking $\mathcal{D}=\emptyset$).

  \item A \emph{good formal model} is a pair $(\mathfrak{Y}, \mathfrak{D})$ consisting of a separated flat regular formal scheme $\mathfrak{Y}$ of topologically finite type over $\cO_K$ and a divisor $\mathfrak{D}\subseteq \mathfrak{Y}$ with simple normal crossings such that $(\mathfrak{Y}_0)_{\rm red}\subseteq \mathfrak{D}$. A morphism of good formal models $\pi\colon (\mathfrak{Y}', \mathfrak{D}')\to (\mathfrak{Y}, \mathfrak{D})$ is a map $\pi\colon \mathfrak{Y}'\to \mathfrak{Y}$ of formal schemes over $\cO_K$ such that $\pi^{-1}(\mathfrak{D})\subseteq \mathfrak{D}'$. This defines the category $\widehat{\mathbf{GM}}_{\cO_K}$ of good formal models. 

  \item The \emph{generic fiber} of a good formal model $(\mathfrak{Y}, \mathfrak{D})$ is the snc pair $(\mathfrak{Y}, \mathfrak{D})_{\rm rig}= (\mathfrak{Y}_{\rm rig}, \mathfrak{D}_{\rm rig})$. The \emph{trivial locus} of a good formal model $(\mathfrak{Y}, \mathfrak{D})$ is the rigid analytic space $(\mathfrak{Y}, \mathfrak{D})_{\rm triv} = \mathfrak{Y}_{\rm rig}\setminus \mathfrak{D}_{\rm rig}$. 

  \item We say that a good formal model $(\mathfrak{Y}, \mathfrak{D})$ is \emph{vertical} if the support of $\mathfrak{D}$ is equal to $\mathfrak{Y}_0$, or equivalently if $(\mathfrak{Y}, \mathfrak{D})_{\rm rig}\in \mathbf{SmRig}_K$. This defines a full subcategory $\widehat{\mathbf{GM}}^{\rm vert}_{\cO_K}$ of $\widehat{\mathbf{GM}}_{\cO_K}$.

  \item A morphism $\pi\colon (\mathfrak{Y}', \mathfrak{D}')\to (\mathfrak{Y}, \mathfrak{D})$ is called an \emph{admissible blowup} if the induced map of snc pairs
  \[ 
    \pi_{\rm rig}\colon (\mathfrak{Y}', \mathfrak{D}')_{\rm rig}\to (\mathfrak{Y}, \mathfrak{D})_{\rm rig}
  \]
  is an isomorphism. We denote the subcategory of $\widehat{\mathbf{GM}}_{\cO_K}$ consisting of admissible blowups again by $W$.
\end{enumerate}
\end{defin}

For rigid snc pairs and smooth rigid analytic spaces, we have the following version of Theorem~\ref{thm:raynaud}. 

\begin{prop} \label{prop:rig-localiz} \mbox{ }
\begin{enumerate}[(a)]
  \item The functor
  \[ 
    (-)_{\rm rig} \colon \widehat{\mathbf{GM}}_{\cO_K}\to \mathbf{RigPairs}_K
  \]
  is faithful and induces an equivalence between the localization $\widehat{\mathbf{GM}}_{\cO_K}[W^{-1}]$ and the full subcategory $\mathbf{RigPairs}^{\rm qcs}_K$ of $\mathbf{RigPairs}_K$ consisting of snc pairs $(\mathcal Y,  \mathcal D)$ such that $\mathcal Y$ is quasi-compact and separated.
  \item The restriction of the above functor
  \[ 
    (-)_{\rm rig} \colon \widehat{\mathbf{GM}}^{\rm vert}_{\cO_K}\to \mathbf{SmRig}_K
  \]
  is faithful and induces an equivalence between the localization $\widehat{\mathbf{GM}}^{\rm vert}_{\cO_K}[W^{-1}]$ and the full subcategory $\mathbf{SmRig}^{\rm qcs}_K$ of $\mathbf{SmRig}_K$ consisting of quasi-compact and separated smooth rigid analytic spaces.
  \item In both cases above, admissible blowups in the respective categories admit a calculus of right fractions. Consequently, $\widehat{\mathbf{GM}}_{\cO_K}[W^{-1}]$ and $\widehat{\mathbf{GM}}^{\rm vert}_{\cO_K}[W^{-1}]$ agree with the respective $\infty$-categorical localizations.
\end{enumerate}
\end{prop}

\begin{proof}
This follows in a standard manner from Lemma~\ref{lemma:detect-localiz}, embedded resolution of singularities, and \cite[Theorem~4.1]{FormalRigidI}. The final assertion follows from Theorem~\ref{thm:dwyer-kan}.
\end{proof}

We can summarize the above discussion with the following commutative diagram of categories and functors.
\[ 
  \xymatrix{
  \mathbf{GM}_{\cO_K} \ar[rr]_{\hphantom{A}L}^-{(-)\otimes K} \ar@{^{(}->}[dd]_{\widehat{(-)}} & & \mathbf{Pairs}^{\rm proper}_K \ar[rr]_-{L}^{(-)_{\rm triv}} \ar@{^{(}->}[dd]_{(-)_\an} & & \mathbf{Sm}_K \ar[dd]^{!}_{(-)_\an}   \\
 \\
  \widehat{\mathbf{GM}}_{\cO_K} \ar[rr]_{\hphantom{A}L}^{(-)_{\rm rig}} & & \mathbf{RigPairs}^{\rm qcs}_K \ar[rr]_{!}^{(-)_{\rm triv}} & & \mathbf{SmRig}_K \\
  \widehat{\mathbf{GM}}_{\cO_K}^{\rm vert} \ar@{^{(}->}[u] \ar[rr]_{\hphantom{A}L}^{(-)_{\rm rig}} & & \mathbf{SmRig}^{\rm qcs}_K. \ar@{^{(}->}[u] \ar@{^{(}->}[urr] & & \\
  }
\]
Here, the horizontal arrows labeled $L$ are localizations in the sense of Definition~\ref{def:localization} and the vertical maps labeled $\hookrightarrow$ are fully faithful functors. The vertical arrow labeled $!$ is not fully faithful, and the horizontal arrow labeled $!$ is not a localization (due to essential singularities). By resolution of singularities, the essential image of the latter functor consists of smooth rigid analytic spaces of the form $\mathcal{X}=\mathcal{Y}\setminus \mathcal{Z}$ where $\mathcal{Y}$ is quasi-compact and separated and $\mathcal{Z}\subseteq \mathcal{Y}$ is a closed analytic subset.

\subsection{Key calculation}\label{sec:key.calculation}

In this section, we analyze the fibers of the map on Kato--Nakayama spaces associated to a simple blowup of complex snc pairs. This will be a key input for showing that the Betti realization will depend only on the generic fiber (Theorem~\ref{thm:simple-blowup-isom}).

Before stating and proving the result, let us take a look at two simple examples. In both, we identify the Kato--Nakayama space $(\mathbf{P}^1_\CC, \infty)_{\rm log}$ with the closed disc $D$ (of ``infinite radius'').

\begin{example} \label{ex:ex1}
  Let $\mathcal{X} = (\mathbf{P}^1_{\CCt})_{\rm an}$ be the rigid projective line over $\CCt$. We choose two good models for $\mathcal{X}$:
  \[
    \mathfrak{X} = (\mathbf{P}^1_{\CCs})^\wedge, \quad
    \mathfrak{X}' = {\rm Bl}_\infty \mathfrak{X},
  \]
  related by an admissible blowup $\pi\colon \mathfrak{X}' \to \mathfrak{X}$ in the point $\infty \in \mathbf{P}^1_\CC$. We equip each of them with the divisor given by the special fiber (suppressed from the notation). We want to analyze the map
  \[ 
    \pi_{0,\rm log}\colon \mathfrak{X}'_{0, \rm log} \to \mathfrak{X}_{0, \rm log}
  \]
  and show that it is a homotopy equivalence.

  In the first case, since $\mathfrak{X}\to {\rm Spf}(\CCs)$ is strict, the space $\mathfrak{X}_{0, \rm log}$ is simply 
  \[
    \mathfrak{X}_{0, \rm log} \simeq \mathbf{P}^1(\mathbf{C}) \times 0_{\rm log} \simeq \mathbf{S}^2 \times \Sone.
  \]
  The special fiber $\mathfrak{X}'_0$ has two components, the strict transform $P\simeq \mathbf{P}^1_\CC$ of $\mathfrak{X}_0$ and the exceptional line $E \simeq \mathbf{P}^1_\CC$, with the $\infty \in P$ identified with $0\in E$. The second space, $\mathfrak{X}'_{0, \rm log}$, is homeomorphic to the space obtained by gluing $P_{\rm log} = (\mathbf{P}^1_\CC, \infty)_{\rm log} \times 0_{\rm log} \simeq D \times \Sone$ and $E_{\rm log} = (\mathbf{P}^1_\CC, 0)_{\rm log} \times 0_{\rm log} \simeq D\times \Sone$  along an identification of the boundary $\partial D$. This again gives $\mathbf{S}^2\times \Sone$, and moreover the map $\pi_{0,\rm log}$ collapses the hemispheres $E_{\rm log} \times \phi$ ($\phi\in \Sone$) into the points $(\infty, \phi) \in \mathbf{P}^1(\mathbf{C})\times \Sone$ (see Figure~\ref{fig:ex12}(a)). Note that this collapsing map is indeed a homotopy equivalence.
\end{example}

\begin{example} \label{ex:ex2}
  We enrich the previous example by adding horizontal boundary. Let $\mathcal{X}, \mathfrak{X}$, and $\mathfrak{X}'$ be as before, and let $\mathcal{D} = \infty \in \mathcal{X}$. We equip $\mathfrak{X}$ and $\mathfrak{X}'$ with the divisors being the union of the special fiber and the closure of $\mathcal{D}$. We again take a closer look at the map
  \[ 
    \pi_{0, \rm log} \colon (\mathfrak{X}', \mathfrak{D}')_{0, \rm log} \to (\mathfrak{X}, \mathfrak{D})_{0, \rm log}. 
  \]
  The space $(\mathfrak{X}, \mathfrak{D})_{0, \rm log}$ is identified with $(\mathbf{P}^1_\CC, \infty)_{\rm log} \times 0_{\rm log} \simeq D \times \Sone$. The component $P$ of $\mathfrak{X}'_0$ has log structure as in the previous example, while the other component $E$ has now non-trivial log structure both at $0$ and $\infty$, and $E_{\rm log} \simeq ([0, \infty] \times \Sone)\times \Sone$. Thus the space $(\mathfrak{X}', \mathfrak{D}')_{0, \rm log}$ is obtained by gluing $D\times \Sone$ with $([0, \infty]\times \Sone)\times \Sone$ along the identification $\partial D \simeq 0\times \Sone$. The map to $D\times \Sone$ is obtained by contracting the interval $[0,\infty]$ to a point (see Figure~\ref{fig:ex12}(b)), and is a homotopy equivalence.
\end{example}

\begin{figure}
    \centering
    \begin{minipage}{0.45\textwidth}
        \centering
        \includegraphics[width=0.9\textwidth]{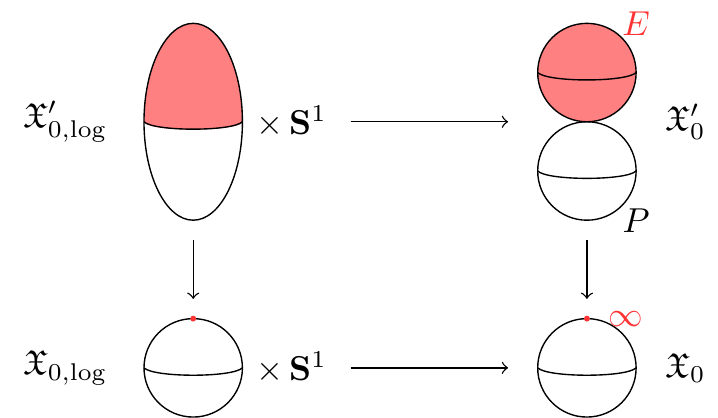} 

        (a) Example~\ref{ex:ex1}
    \end{minipage}\hfill
    \begin{minipage}{0.45\textwidth}
        \centering
        \includegraphics[width=0.9\textwidth]{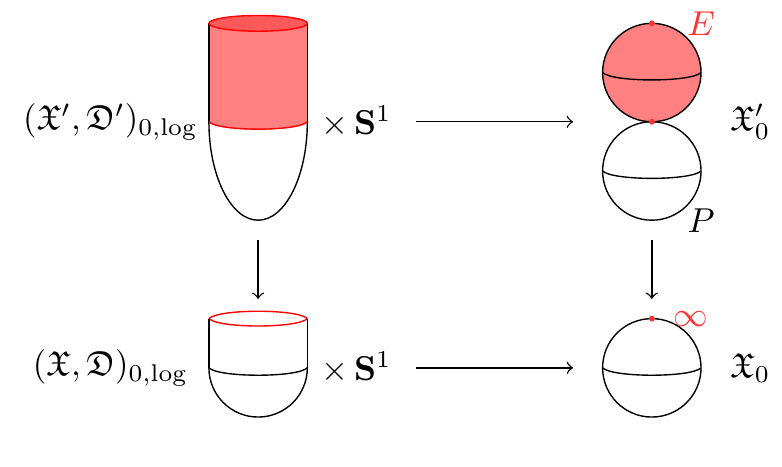} 
        
        (a) Example~\ref{ex:ex2}
    \end{minipage}
    \caption{Maps on Kato--Nakayama spaces of special fibers induced by some simple blowups.}
    \label{fig:ex12}
\end{figure}

\begin{prop}\label{prop:key-calculation}
Let $X$ be a complex manifold, $D\subseteq X$ a divisor with normal crossings, ${Z\subseteq X}$ a closed smooth submanifold of $X$ which has normal crossings with $D$. Let $\pi\colon X'\to X$ be the blowing-up of $X$ along $Z$. Endow $X$ (resp. $X'$) with the log structure induced by the open subset $U=X\setminus D$ (resp. $U'=\pi^{-1}(U)$). Then $\pi$ extends uniquely to a map of log schemes $\pi\colon (X', \Mm_{X'\setminus U'})\to (X, \Mm_{X\setminus U})$ and the fibers of the induced map of Kato--Nakayama spaces
\[ 
  \pi_{\rm log} \colon  (X', \Mm_{X'\setminus U'})_{\rm log} \to (X, \Mm_{X\setminus U})_{\rm log}
\]
are contractible and locally contractible.
\end{prop}

\begin{proof}
The question is local on $X$, and hence we might assume that we are in the following situation. Fix a finite set $I$, and let $B, L\subseteq I$ be two subsets with $B\cap L\neq \emptyset$. Let $X$ be an open neighborhood of the origin $0\in \CC^I$, let $U\subseteq X$ be the open subset defined by the non-vanishing of the coordinates $x_i$ ($i\in L$), and let $Z\subseteq X$ be the closed subspace defined by the vanishing of the coordinates $x_i$ ($i\in B$). 

We can assume that $B=I$: otherwise, by shrinking $X$ we can assume that it is of the form $X' \times X''$ where $X', X''$ are open neighbourhoods of the origin in $\CC^B$ and $\CC^{I\setminus B}$ respectively. The blowing-up leaves the second factor untouched, so we are reduced to analyzing what happens in the first factor.

We have $X' = V(x_i Y_j - x_j Y_i)\subseteq X \times \PP^I$ and $D' = X'\setminus U' = \bigcup_{i\in L} D_i \cup Y$ where $Y=\pi^{-1}(0)$ and $D_i = V(Y_i)$. Here $\PP^I$ denotes the projective space with homogeneous coordinates $Y_i$ ($i\in I$). This gives a chart $\NN^L\times \NN\cdot e \to {\rm Div}({X'})$ for the log structure of $X'$, sending $\ell$ to $(\cO(1), Y_\ell)$ and $e$ to $(\cO_{X'}(Y), 1_Y)$. Here we are using the language of Deligne--Faltings structures (introduced in \cite{borne-vistoli}), according to which a log structure can be seen as a symmetric monoidal functor $L\colon A \to \Div_X$ from a sheaf of monoids $A$ to the symmetric monoidal stack of pairs $(L,s)$, where $L$ is a line bundle and $s$ is a section of $L$. A chart is then a symmetric monoidal functor $P\to \Div(X)$ where $P$ is a monoid, inducing the functor $L$ via a sheafification procedure. For simplicity, we will use the same notation for schemes and log schemes below.

Now note that $\cO_{X'}(Y) \simeq \cO(-1)$. If we restrict this chart to the exceptional divisor $Y\simeq \PP^I$, we obtain a symmetric monoidal functor $\NN^L\times \NN\cdot e \to {\rm Div}({Y})$, sending $\ell\mapsto (\cO_Y(1), Y_\ell)$ and $e\mapsto (\cO_Y(-1), 0)$. We have a chart for the map $\pi\colon Y\to 0=\Spec (\NN^L\to \CC)$:
\[ 
  \xymatrix{
    \NN^L\times \NN\cdot e \ar[r] & {\rm Div}(Y) \\
    \NN^L \ar[u]^{\ell \mapsto \ell + e} \ar[r] & {\rm Div(0)}. \ar[u]
  }
\]
Write $Y = \PP^I = (\CC^I \setminus\{0\})/\CC^\times$. Note that this is a quotient presentation of $Y$ as a log scheme, and the action of $\CC^\times$ is free. From this, and the discussion of how the Kato--Nakayama construction behaves with respect to quotients of \cite[Section 4]{knroot}, we obtain an isomorphism
\[
Y_{\rm log}= (\CC^I \setminus\{0\})_\log/\CC^\times\simeq  \left(\left(\left(\overline \CC^L \times \CC^{I\setminus L}\right) \setminus \left( (\mathbf{S}^1\times 0)^L \times 0^{I\setminus L}\right)\right)\times \mathbf{S}^1\right) / \CC^\times.
\]
Here $\overline{\CC}=\RR_{\geq 0}\times \Sone$, as in Section \ref{sec:knreview}, and $\overline \CC^L \times \CC^{I\setminus L}$ is the Kato--Nakayama space of $\CC^I$ equipped with the log structure given by the coordinates with indices in $L$, from which we have to subtract the preimage of $\{0\}_\log$, i.e. the subset $(\mathbf{S}^1\times 0)^L \times 0^{I\setminus L}$. The last $\Sone$ factor corresponds to the generator $e$ of the chart    $\NN^L\times \NN\cdot e$.

Now if the coordinates are $(r_i, \phi_i)$ on $\overline\CC^L$, $z_i$ on $\CC^{I\setminus L}$, and $\phi$ on the $\mathbf{S}^1$, the action of $\lambda\in \CC^\times$ is given by the formula
\[ 
  \lambda((r_i, \phi_i)_{i\in L}, (z_i)_{i\notin L}, \phi)
   = \left(\left(|\lambda|r_i, \frac{\lambda}{|\lambda|}\phi_i\right), \lambda z_i, \left(\frac{\lambda}{|\lambda|}\right)^{-1}\phi\right).
\]
With this description of $Y_\log$, the map $\pi_{0, {\rm log}} \colon  Y_{\rm log}\to 0_{\rm log}$ takes the form of the map
\[
\left(\left(\overline \CC^L \times \CC^{I\setminus L} \setminus (\mathbf{S}^1\times 0)^L \times 0^{I\setminus L}\right)\times \mathbf{S}^1\right) / \CC^\times\to (\Sone)^L
\]
given by the formula
\[
 [((r_i,\phi_i)_{i\in L}, (z_i)_{i\notin L},\phi)]\mapsto \phi\cdot \phi_i.
\]
Note that this is well-defined because of the way $\lambda$ acts on $\phi_i$ and $\phi$.

Now let us finally analyze one fiber of this map, i.e.\ fix the values of $(\psi_i)_{i\in L}\in (\Sone)^L$. The circle $\mathbf{S}^1\subseteq\CC^\times$ acts freely on the coordinate $\phi$, and
\begin{align*}
  \pi_{0, {\rm log}}^{-1}((\psi_i)_{i\in L})
  &\simeq \left(\RR^L_{\geq 0} \times \CC^{I\setminus L} \setminus \{0\}\right) / \RR_{>0} \\
  &\simeq \left\{ \sum_{i\in L} r_i^2 + \sum_{i\notin L} |z_i|^2 = 1\right\} \subseteq \RR_{\geq 0}^L \times \CC^{I\setminus L} \\
  &\simeq (\text{a convex cone }\neq\text{ whole of }\RR^L\times \CC^{I\setminus L})\cap (\text{unit sphere}),
\end{align*}
which is contractible (and locally contractible). Note that the final step used the fact that $L$ (that is, $L\cap B$) is nonempty, which corresponds to the fact that $Z\subseteq D$.
\end{proof}

\subsection{Construction of the functor (I): Weak factorization}
\label{sec:construction-rigid}

\begin{defin}\label{def:simple.bl}
  A \emph{simple blowup} is a map $\pi\colon (\mathfrak{Y}', \mathfrak{D}')\to (\mathfrak{Y}, \mathfrak{D})$ in $\widehat{\mathbf{GM}}_{\cO_K}$ such that $\pi\colon \mathfrak{Y}'\to \mathfrak{Y}$ is the blowup of $\mathfrak{Y}$ along a closed regular formal subscheme $\mathfrak{Z}\subseteq \mathfrak{X}$ which has \emph{normal crossings} with $\mathfrak{D}$, i.e. in formal local coordinates $\mathfrak{D}=\{x_1\cdots x_r = 0\}$ and $\mathfrak{Z} = \{x_i = 0 \,|\, i\in I\}$ for some index set $I$. An \emph{admissible simple blowup} is a simple blowup which is also admissible, i.e.\ if $\mathfrak{Z}$ is supported in $\mathfrak{Y}_0$.
\end{defin}

\begin{thm}[{\cite[Theorem  6.4.5]{weakfactorization}}] \label{thm:wft}
  The class of admissible simple blowups has the weak factorization property (Definition~\ref{def:weak.fact}) in the category with equivalences $(\widehat{\mathbf{GM}}_{\cO_K}, W)$.
\end{thm}

Combining this with Lemma~\ref{lemma:weak.fact}, we obtain:

\begin{cor}
  Every functor $\widehat{\mathbf{GM}}_{\cO_K}\to \mathscr{D}$ to an $\infty$-category $\mathscr{D}$ sending simple blowups to equivalences factors uniquely through the category $\mathbf{RigPairs}^{\rm qcs}_K$.\qed
\end{cor}

Consider the functor
\[ 
  \Psi_{\rm log} \colon \widehat{\mathbf{GM}}_{\cO_K} \to  \SpSone, 
  \quad
  \Psi_{\rm log}(\mathfrak{Y}, \mathfrak{D}) = \mathfrak{Y}_{0, \rm log}.
\]

\begin{thm} \label{thm:simple-blowup-isom}
  The functor $\Psi_{\rm log}$ sends admissible simple blowups to equivalences.
\end{thm}

In the proof of this result we will make use of the following theorem:

\begin{thm}[\cite{smale}] \label{thm:smale}
  Let $S$ and $T$ be connected and locally compact separable metric spaces, with $S$ locally contractible. Let $f\colon S\to T$ be a proper surjective continuous map whose fibers $f^{-1}(t)$ are contractible and locally contractible. Then $f$ is a weak homotopy equivalence. \qed
\end{thm}

\begin{proof}[Proof of Theorem~\ref{thm:simple-blowup-isom}]
Let $\pi\colon (\mathfrak{Y}', \mathfrak{D}') \to (\mathfrak{Y}, \mathfrak{D})$ be an admissible simple blowup between good formal models, the blowup at a $\mathfrak{Z}\subseteq\mathfrak{Y}$ as in Definition~\ref{def:simple.bl}. We are going to prove that the fibers of the induced morphism 
\[ 
  \Psi_{\rm log}(\pi) = \pi_{0, \rm log}\colon (\mathfrak{Y}', \mathfrak{D}')_{0, \rm log} \to (\mathfrak{Y}, \mathfrak{D})_{0, \rm log}
\] 
are contractible and locally contractible. We can assume that $K=\CCt$.

We fix a closed point $y\in \mathfrak{Y}_0$ on the central fiber $\mathfrak{Y}_0$ and local coordinates $x_1, \ldots, x_n$ of $\mathfrak{Y}$ at $y$ in which $t = \prod x_i^{a_i}$ for some integers $a_i$, the divisor $\mathfrak{D}$ is defined by $x_1\cdots x_r$, and the center $\mathfrak{Z}$ is defined by the vanishing of $x_i$ for $i$ in some subset $I\subseteq \{1, \ldots, n\}$. Let $\mathbf{A}^n = \Spec \CC[z_1, \ldots, z_n]$ with the log structure given by the divisor $D = \{z_1\cdots z_r = 0\}$, and let $B$ be its blowup at the ideal $(z_i\,|\,i\in I)$, with the log structure given by the reduced preimage of $D$. After passing to an \'etale neighborhood of $y$, the local coordinates produce a cartesian diagram of fs log schemes of finite type over $\CC$
\[ 
  \xymatrix{
    (\mathfrak{Y}'_0, \Mm_{\mathfrak{Y}'_0}) \ar[r] \ar[d]_{\pi_0} & B \ar[d] \\
    (\mathfrak{Y}_0, \Mm_{\mathfrak{Y}_0}) \ar[r] & \mathbf{A}^n
  }
\]
where the horizontal maps are strict. Passing to Kato--Nakayama spaces preserves the cartesian property of the last diagram, and we conclude that every fiber of $\pi_{0, \rm log}$ is also a fiber of $B_\log\to \AA^n_\log$. All these fibers are contractible and locally contractible by Proposition \ref{prop:key-calculation}, and using Smale's theorem (Theorem~\ref{thm:smale}) we can conclude that $\Psi_{\rm log}(\pi)$ is an equivalence (note that the map $\pi_{0, \rm log}$ is clearly compatible with the maps to $\Sone$).
\end{proof}

This allows us to define the desired functor $\Psi_{\rm rig}$.

\begin{cor}
  The functor $\Psi_{\rm log}$ extends uniquely to a functor
  \[ 
    \Psi_{\rm rig} \colon \mathbf{RigPairs}_K^{\rm qcs} \to \SpSone.  \vspace{-.57cm}
      \]  \qed  \vspace{.1cm}
\end{cor}

For analytifications of smooth schemes over $K$, this new functor agrees with the functor $\Psi$ constructed in \S\ref{sec:spreading}--\ref{sec:log.geometry} in the following sense.

\begin{prop} \label{prop:rig-comp-1}
    The following diagram commutes
    \[ 
      \xymatrix{
        \mathbf{Sch}^{\rm ft}_K  \ar[dr]^{\hphantom{aaaa}\Psi \text{ constructed in \S\ref{sec:spreading}--\ref{sec:log.geometry}}} \\
        \mathbf{Pairs}^{\rm proper}_K \ar[u]^{(-)_{\rm triv}} \ar[d]_{(-)_\an} & \SpSone \\ 
        \mathbf{RigPairs}^{\rm qc}_K . \ar[ur]_{\hphantom{aaaa}\Psi_{\rm rig}\text{ defined above}}
      }
    \]
\end{prop}

\begin{proof}
This is clear since the functors coincide on $\mathbf{GM}_{\cO_K}$, see the discussion following Proposition~\ref{prop:rig-localiz}.
\end{proof}

\subsection{Construction of the functor (II): Descent and purity}
\label{sec:descent-rigid}

In order to extend the functor $\Psi_{\rm rig}$ defined above to all rigid snc pairs, we will check that it has the homotopy cosheaf property with respect to the strict admissible topology, i.e.\ the topology on $\mathbf{RigPairs}_K$ given by families $\{(\mathcal{Y}_i, \mathcal{D}_i)\to (\mathcal{Y}, \mathcal{D})\}$ where $\{\mathcal{Y}_i\to \mathcal{Y}\}$ is an admissible covering and where $\mathcal{D}_i = \mathcal{D}\times_{\mathcal{Y}} \mathcal{Y}_i$. 

\begin{prop}
  The functor $\Psi_{\rm rig}\colon \mathbf{RigPairs}_K^{\rm qcs}\to \SpSone$ satisfies descent with respect to strict admissible coverings.
\end{prop}

\begin{proof}
We begin by checking that the functor associating to a good formal model $(\mathfrak{Y}, \mathfrak{D}) \in \widehat{\bf GM}_{\cO_K}$ the Kato--Nakayama space of its special fiber $(\mathfrak{Y}, \mathfrak{D})_{0, \rm log}$ satisfies hyperdescent with respect to the strict Zariski topology (where coverings are given by Zariski coverings of $\mathfrak{Y}$ equipped with the restriction of $\mathfrak{D}$). Indeed, if 
\[
  (\mathfrak{Y'}_\bullet, \mathfrak{D'}_\bullet) \to (\mathfrak{Y}, \mathfrak{D})
\]
is a strict Zariski hypercovering, then 
\[
  (\mathfrak{Y}'_\bullet, \mathfrak{D}'_\bullet)_{0, \rm log} \to (\mathfrak{Y}, \mathfrak{D})_{0, \rm log}
\]
is a hypercovering of topological spaces. By \cite{DuggerIsaksen}, the induced map 
\[
  \hocolim \; (\mathfrak{Y}'_\bullet, \mathfrak{D}'_\bullet)_{0, \rm log} \to (\mathfrak{Y}, \mathfrak{D})_{0, \rm log}
\]
is an equivalence.

Given the above Zariski sheaf property, it is enough to note the following: every \emph{finite} strict admissible covering 
\[
  \{(\mathcal{Y}_i, \mathcal{D}_i) \to (\mathcal{Y}, \mathcal{D}) \}
\]
in $\mathbf{RigPairs}_K^{\rm qcs}$ admits a model
\[
  \{(\mathfrak{Y}_i, \mathfrak{D}_i) \to (\mathfrak{Y}, \mathfrak{D}) \}
\]
which is a strict Zariski covering in $\widehat{\bf GM}_{\cO_K}$. This follows from \cite[Lemma~4.4]{FormalRigidI} and resolution of singularities.
\end{proof}

\begin{cor} \label{cor:rigdescent}
  There exists a unique extension 
  \[
    \Psi_{\rm rig}\colon \mathbf{RigPairs}_K \to \SpSone
  \]
  of the functor $\Psi_{\rm rig}\colon \mathbf{RigPairs}_K^{\rm qcs}\to \SpSone$ which satisfies descent with respect to admissible coverings. \qed
\end{cor}

Our fundamental result below states that $\Psi_{\rm rig}((\mathcal{Y}, \mathcal{D})) = \Psi_{\rm rig}(\mathcal{Y}\setminus \mathcal{D})$. Note that $\mathcal{Y}\setminus\mathcal{D}$ will rarely be quasi-compact, so Corollary~\ref{cor:rigdescent} is used in the statement. 

\begin{thm}[Purity] \label{thm:purity}
  Let $(\mathcal{Y}, \mathcal{D})$ be an object of $\mathbf{RigPairs}^{\rm qcs}_K$ and let $\mathcal{X} = \mathcal{Y}\setminus \mathcal{D}$. Then the inclusion $j\colon \mathcal{X}\to (\mathcal{Y}, \mathcal{D})$ in $\mathbf{RigPairs}_K$ induces an equivalence
  \[ 
    \Psi_{\rm rig}(\mathcal{X}) \isomlong  \Psi_{\rm rig}(\mathcal{Y}, \mathcal{D}).
  \]
\end{thm}

\begin{proof}
By admissible descent (Corollary~\ref{cor:rigdescent}), it is enough to construct a covering family $\{\mathcal{X}_i\}_{i\in I}$ of quasi-compact opens of $\mathcal{X}$ indexed by a filtering poset $I$ with the property that 
\[ 
  \Psi_{\rm rig}(\mathcal{X}_i) \to  \Psi_{\rm rig}(\mathcal{Y}, \mathcal{D}).
\]
is an equivalence for each $i\in I$. To this end, let $I$ be \emph{opposite} of the cofiltering poset of models of $(\mathcal{Y}, \mathcal{D})$, i.e.\ the category of good formal models $(\mathfrak{Y}, \mathfrak{D})$ equipped with an isomorphism $(\mathfrak{Y}, \mathfrak{D})_{\rm rig}\simeq (\mathcal{Y}, \mathcal{D})$. 

For each $i\in I$ corresponding to a good formal model $(\mathfrak{Y}_i, \mathfrak{D}_i)$, let 
\[ 
  {\rm sp}_i \colon \langle\mathcal{Y} \rangle  \to |\mathfrak{Y}_{i,0}|
\]
denote the specialization map. Here $\langle\mathcal{Y} \rangle = \varprojlim |\mathfrak{Y}_{i,0}|$ is the Riemann--Zariski space of $\mathcal{Y}$, see \cite[II \S 3]{FujiwaraKato}. We let
\[
  U_i = \mathfrak{Y}_{i,0}  \setminus {\rm sp}_i(\mathcal{D})  
  \quad\text{and} \quad \mathcal{X}_i = {\rm sp}_i^{-1}(U_i) \subseteq \mathcal{X}
\]
be the set of points specializing to the ``vertical part'' of $\mathfrak{Y}_{i,0}$. This is a quasi-compact open with a vertical good formal model $(\mathfrak{Y}_i, \mathfrak{D}_i)|_{U_i}$, and the inclusion $(\mathfrak{Y}_i, \mathfrak{D}_i)|_{U_i} \to (\mathfrak{Y}_i, \mathfrak{D}_i)$ is a model of the open immersion $\mathcal{X}_i\hookrightarrow (\mathcal{Y}, \mathcal{D})$ in $\mathbf{RigPairs}_K$. 

If $i\leq j$ in $I$, in other words if there exists a morphism $\pi\colon (\mathfrak{Y}_j, \mathfrak{D}_j) \to (\mathfrak{Y}_i, \mathfrak{D}_i)$, we have a~commutative 
\[
  \xymatrix{
    \langle\mathcal{Y}\rangle \ar[r]^{{\rm sp}_j} \ar[dr]_{{\rm sp}_i} & |\mathfrak{Y}_{j,0}| \ar[d]^\pi \\  
    & |\mathfrak{Y}_{i,0}|,
  }
\]
and therefore get an inclusion $\mathcal{X}_i\subseteq \mathcal{X}_j$ (note the opposite direction). 

Clearly, $\mathcal{X}_i \subseteq\mathcal{X}$ for all $i$. We check that the $\mathcal{X}_i$ cover $\mathcal{X}$. To this end, let $\mathcal{W}\subseteq \mathcal{X}$ be any quasi-compact open. By \cite[Lemma~4.4]{FormalRigidI}, there exists a model of the inclusion $\mathcal{W}\subseteq \mathcal{X}$  as an open immersion of models $j\colon \mathfrak{W} \hookrightarrow \mathfrak{Y}$. Since good models form a cofinal system of models of $\mathcal{Y}$, we can assume that $\mathfrak{Y} = \mathfrak{Y}_i$ for some $i\in I$, so $\mathcal{W} = {\rm sp}_i^{-1}(\mathfrak{W}_0)$. But then $\mathcal{W}$ is contained in $\mathcal{X}_i$, for otherwise $\mathfrak{W}_0$ would contain a point in ${\rm sp}_i(\mathcal{D})$, and so $\mathcal{W}$ would intersect $\mathcal{D}$.

Finally, to check that $\Psi_{\rm rig}(\mathcal{X}_i)\to \Psi_{\rm rig}(\mathcal{Y}, \mathcal{D})$ is an equivalence, we need to show that 
\[ 
  U_{i,{\rm log}} \to \mathfrak{Y}_{i,0, {\rm log}},
\]
is an equivalence, where we endow $U_i$ with the log structure induced from $\mathfrak{Y}_{i,0}$. This follows from Lemma~\ref{lemma:kn-vert} below.
\end{proof}

\begin{lemma} \label{lemma:kn-vert}
  Let $(X, \Mm_X)$ be a smooth log scheme over the standard complex log point $0$, where $X$ is separated and of finite type over $\CC$. Denote by $X^{\rm vert}\subseteq X$ the dense open subset where $(X, \Mm_X)$ is vertical (i.e., where $\overline{\Mm}_0\simeq \mathbf{N}$ generates $\overline{\Mm}_X$ as a face). Then the induced map
  \[ 
    (X^{\rm vert}, \Mm_X|_{X^{\rm vert}})_{\rm log}
    \to (X, \Mm_X)_{\rm log}
  \]
  is a homotopy equivalence.
\end{lemma}

\begin{proof}
  Since the base is the standard log point, $(X, \Mm_X)\to 0$ is automatically exact. By \cite[Theorem~3.7]{nakayama-ogus}, the map $(X, \Mm_X)_{\rm log} \to 0_{\rm log} = \Sone$ is a topological submersion whose fibers are manifolds with boundary equal to the preimage of $X\setminus X^{\rm vert}$. The inclusion of the complement of the boundary in a manifold with boundary is a homotopy equivalence.
\end{proof}

We can now improve our comparison result Proposition~\ref{prop:rig-comp-1}.

\begin{cor} \label{cor:rig-comp-2}
  The functors $\Psi$ and $\Psi_{\rm rig}$ are compatible in the sense that the following diagram commutes
  \[ 
    \xymatrix{
      \mathbf{Sm}^{\rm lft}_K \ar[d]_{(-)_\an} \ar[dr]^{\Psi} \\
      \mathbf{SmRig}_K \ar[r]_-{\Psi_{\rm rig}} & \SpSone. 
    } \vspace{-.7cm}
      \] \qed \vspace{.1cm}
\end{cor}

\subsection{Examples and open questions} 
\label{ss:examples}

We will now discuss two quite classical examples: the Tate curve and the non-archimedean Hopf surface.

\begin{example}[Tate curve]
  The Tate curve is the elliptic curve
  \[ 
    \mathcal{X} = \mathcal{Y}/q^\ZZ, \quad \mathcal{Y} = (\mathbf{G}_m)_\an \text{ over }\CCq.
  \]
  Recall the construction of a model of $\mathcal{X}$ (e.g.\ \cite[\S 9.2]{Bosch}). We start with a model for $\mathcal{Y}$. For $n\in\ZZ$, set 
  \[
    \mathfrak{Y}(n) = \Spf \CCsq\{x_n, y_n\}/(x_n y_n - q).
  \]
  
  Let $\mathfrak{Y}^+(n) = \Spf \CCsq\{x_n,x_n^{-1}\}$ and $\mathfrak{Y}^-(n) = \Spf \CCsq\{y_n, y_n^{-1}\}$, which are open formal subschemes of $\mathfrak{Y}(n)$. We form the formal scheme $\mathfrak{Y}$ by taking the disjoint union of all $\mathfrak{Y}(n)$ and identifying $\mathfrak{Y}^+(n)$ with $\mathfrak{Y}^-(n+1)$ via $y_{n+1}\mapsto x_n^{-1}$. The formal scheme $\mathfrak{Y}$ is a model for $\mathcal{Y} = (\mathbf{G}_m)_\an$. The map $q\colon \mathcal{Y}\to \mathcal{Y}$ has a model $q\colon \mathfrak{Y}\to \mathfrak{Y}$ sending $\mathfrak{Y}(n)$ to $\mathfrak{Y}(n+1)$ via $x_{n+1}\mapsto x_n$, $y_{n+1}\mapsto y_n$. Let $\Xx$ be the quotient of $\mathfrak{Y}$ by the action of $q^\ZZ$. Its generic fiber 
  \[
    \mathcal{X} = \Xx_{\rm rig} = (\mathbf{G}_m)_\an/q^\ZZ
  \]
  is the analytification of an elliptic curve $X$ over $\CCq$ called the Tate curve. 

  We will now investigate the log structures of $\mathfrak{Y}$ and $\mathfrak{X}$. Let $\NN \langle X, Y\rangle$ be the free commutative monoid with generators $X$ and $Y$. A chart for $\mathfrak{Y}(n)$ over the standard chart $\NN\langle Q \rangle\to \CCsq$, $Q\mapsto q$ is given by 
  \[ 
    \NN \langle X, Y\rangle\to \CCsq\{x_n, y_n\}/(x_n y_n - q), \quad Q \mapsto XY, \quad X\mapsto x_n,\quad Y\mapsto y_n.
  \]
  Thus $\mathfrak{Y}(n)$ is a good model. Following \cite[\S 4.1]{achinger-ogus}, we can identify $\mathfrak{Y}(n)_{0,\rm log}$ with the union of two copies of $\Sone \times \overline \CC$ (coordinates $(\tau, \overline x_n)$ and $(\tau, \overline y_n)$ where $\tau = \arg q$) where the two copies of $\Sone\times\Sone$ are identified via the map $(\tau, \alpha) \mapsto (\tau, \tau \alpha^{-1})$. See Figure~\ref{fig:tate} for an illustration of $\mathfrak{Y}_{0, \rm log}$.

  \begin{figure}[h]
    \centering
    \includegraphics[width=.5\textwidth]{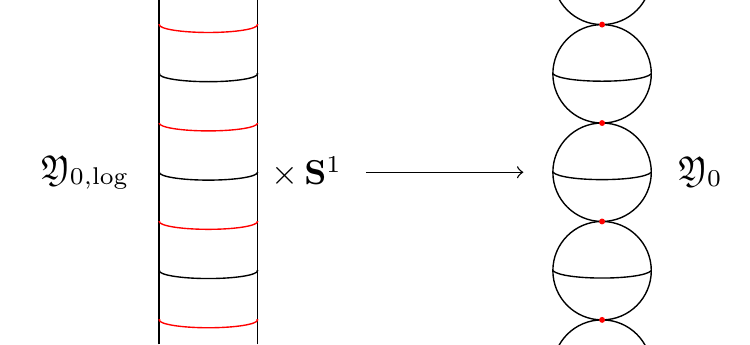} 
    \caption{Tate curve: the Kato--Nakayama space of the special fiber of $\mathfrak{Y}$.} 
    \label{fig:tate}
  \end{figure}

  Consequently, we see that $\mathfrak{X}_{0,\rm log} = \mathfrak{Y}_{0,\rm log}/q^\ZZ$ can be described as the quotient of 
  \[ 
    \Sone \times (\PP^1, \{0,\infty\})_{\rm log} = \Sone\times ([0,\infty]\times \Sone)
  \]
  by the relation
  \[ 
    (\tau, 0, \alpha) \sim (\tau, \infty, \tau\alpha).
  \]
  Thus $\Psi_{\rm rig}(X)$ is the family of real tori $\Sone \times \Sone$ over $\Sone$ with monodromy a single Dehn twist. 
\end{example}

\begin{example}[Nonarchimedean Hopf surface]
  The classical Hopf surface is the compact complex manifold obtained as a quotient $(\mathbf{C}^2 \setminus 0)/q^{\ZZ}$ where $|q|>1$ is a complex number acting by scalar multiplication. It is diffeomorphic to $\Sone\times \mathbf{S}^3$, and therefore does not admit a symplectic structure, nor does it admit a Hodge decomposition.

  Nonarchimedean analogues of Hopf surfaces were studied by several authors; we will use the description given by Voskuil \cite{voskuil}. For simplicity, let us consider the case where $q=t$ is the uniformizer, i.e.
  \[ 
    \mathcal{X} = \mathcal{Y} / q^\ZZ, \quad \mathcal{Y} = (\mathbf{A}^2_\CCq \setminus 0)_\an.
  \]
  We claim that the fiber $\widetilde\Psi_{\rm rig}(\mathcal{X})$ of $\Psi_{\rm rig}(\mathcal{X})/\Sone$ is homotopy equivalent to $\Sone\times \mathbf{S}^3$ as well, in fact we will find that the Kato--Nakayama space of a certain model of $\mathcal{X}$ is a fibration over $\Sone$ with fiber homeomorphic to $\Sone\times \mathbf{S}^3$.

  The rigid space $\mathcal{Y}$ admits a good formal model $\mathfrak{Y}$ whose special fiber can be described as follows. Let $P\in \mathbf{P}^2_\CC$ be a point and let $Z = \mathrm{Bl}_P \mathbf{P}^2_\CC$ be the corresponding blowup; we denote by $\mathbf{P}^1_\CC \simeq E\subseteq Z$ the exceptional divisor. Let $\mathbf{P}^1_\CC\simeq L\subseteq Z$ be the strict transform of a line in $\mathbf{P}^2_\CC$ not passing through $P$. There is a natural identification $\iota\colon E\simeq L$, sending a point $Q\in L$ to the point of intersection of the strict transform of the line $PQ$ with $E$. Then $\mathfrak{Y}_0$ is the union of isomorphic components $Z_n \simeq Z$ indexed by $n\in\ZZ$, where $E_n\subseteq Z_n$ is identified with $L_{n+1}\subseteq Z_{n+1}$ using $\iota$. The action of $t^\ZZ$ extends to an action on $\mathfrak{Y}$, and the action on the special fiber is just the translation $t\cdot Z_n = Z_{n+1}$. 

  It follows that $\mathcal{X}$ admits a good formal model $\mathfrak{X}$ whose special fiber is the non-projective surface $Z / (E\sim L)$ with normal crossings (\cite[Proposition~2.5]{voskuil} with $m=1$). Its Kato--Nakayama space for the natural log structure can be described as the quotient of 
  \[ 
    \Sone \times (Z, E+L)_{\rm log} \simeq \Sone \times ([0, \infty]\times \mathbf{S}^3)
  \]
  by the relation
  \[ 
    (\tau, 0, x) \sim (\tau, \infty, \tau\cdot x).
  \]
  Thus the fiber of $\mathfrak{X}_{0,\log}\to \Sone$ over $\tau = 1$ is $[0, \infty]/(0\sim \infty) \times \mathbf{S}^3\simeq \Sone\times \mathbf{S}^3$.
\end{example}

We finish with two natural questions which we were unable to answer.

\begin{question} \label{question:descent}
  Does $\Psi_{\rm rig}$ satisfy descent with respect to proper hypercoverings?
\end{question}

If this would be the case, then using resolution of singularities one could extend the definition of $\Psi_{\rm rig}$ to rigid spaces which are not necessarily smooth.

\begin{question}
  Is $\Psi_{\rm rig}(X)$ well-defined up to homeomorphism?
\end{question}

This is the case for schemes, as seen using the spreading out construction of \S\ref{sec:spreading}. More precisely, this question is asking whether for two good formal models $(\mathfrak{Y}_i, \mathfrak{D}_i)$ ($i=1,2$) with the same generic fiber, there exists a (non-canonical) homeomorphism
\[
  (\mathfrak{Y}_1, \mathfrak{D}_1)_{0, \rm log} \simeq (\mathfrak{Y}_2, \mathfrak{D}_2)_{0, \rm log}
\]  
over $\Sone$, whose homotopy class is the homotopy equivalence used in the construction of $\Psi_{\rm rig}$. In Examples~\ref{ex:ex1} and \ref{ex:ex2}, we saw that in simple cases, maps induced by simple blowups of good models are in some sense ``trivial surgeries,'' and while they are not homeomorphisms themselves, they are homotopic to homeomorphisms.

\bibliographystyle{amsalpha} 
\bibliography{bib.bib}

\end{document}